\documentclass[12pt,twoside]{amsart}
\usepackage{amssymb,amsmath,amsthm, amscd, enumerate, mathrsfs}
\usepackage{graphicx, hhline}
\usepackage[all]{xy}
\usepackage[usenames]{color}
\usepackage{hyperref}
\hypersetup{colorlinks=true}

\title{Notes on the weak positivity theorems} 
\author{Osamu Fujino} 
\date{2015/6/30, version 0.54}
\keywords{semipositivity theorem, weak positivity, mixed Hodge structures, 
Iitaka's conjecture}
\subjclass[2010]{Primary 14J10; Secondary 14D07, 32G20}
\address{Department of Mathematics, Graduate School of Science, 
Kyoto University, Kyoto 606-8502, Japan}
\email{fujino@math.kyoto-u.ac.jp}

\newcommand{\codim}[0]{{\operatorname{codim}}}
\newcommand{\GL}[0]{{\operatorname{GL}}}
\newcommand{\Cone}[0]{{\operatorname{Cone}}}
\newcommand{\Spec}[0]{{\operatorname{Specan}}}
\newcommand{\Pic}[0]{{\operatorname{Pic}}}
\newcommand{\Exc}[0]{{\operatorname{Exc}}}
\newcommand{\Supp}[0]{{\operatorname{Supp}}}
\newcommand{\Gr}[0]{{\operatorname{Gr}}}
\newtheorem{thm}{Theorem}[section]
\newtheorem{lem}[thm]{Lemma}
\newtheorem{cor}[thm]{Corollary}
\newtheorem{prop}[thm]{Proposition}
\newtheorem{conj}[thm]{Conjecture}

\theoremstyle{definition}
\newtheorem{defn}[thm]{Definition}
\newtheorem{rem}[thm]{Remark}
\newtheorem*{ack}{Acknowledgments} 
\newtheorem{say}[thm]{}
\newtheorem{step}{Step}

\begin{document}

\maketitle 

\begin{abstract}
We discuss the (twisted) weak positivity 
theorem. We also treat some applications. 
\end{abstract}

\tableofcontents 

\section{Introduction}\label{f-sec1} 

In this paper, we discuss the (twisted) weak positivity theorem. 
We give a detailed proof of the following theorem, which is 
essentially equivalent to \cite[Theorem 4.13]{campana} (see also \cite{lu}). 
The proof is based on our semipositivity theorem 
(see Theorem \ref{f-thm1.5}, \cite{fujino1}, \cite{fujino-fujisawa}, and \cite{ffs}). 
Note that Theorem \ref{f-thm1.1} has already played important roles 
in \cite{hacon-mckernan}, \cite{fujino-gongyo}, and so on, when 
$X$ is projective (see also \cite{kovacs-patakfalvi}). 

\begin{thm}[Twisted weak positivity]\label{f-thm1.1}  
Let $(X, \Delta)$ be a log canonical 
pair such that $X$ is in Fujiki's 
class $\mathcal C$ and let $f:X\to Y$ be a surjective morphism 
onto a smooth projective variety $Y$. Assume that 
$k(K_X+\Delta)$ is Cartier. 
Then, for every positive integer $m$, 
 $$f_*\mathcal O_X(mk(K_{X/Y}+\Delta))$$ is weakly positive. 
 \end{thm}

We have already discussed 
some generalizations of Theorem \ref{f-thm1.1} in \cite{fujino4}, 
where $Y$ is a curve and $X$ is a reducible variety. 
They play crucial roles in order to prove the projectivity of 
various moduli spaces. For the details, see \cite{fujino4} and \cite{kovacs-patakfalvi}. 
For the basic properties of weakly positive sheaves and 
Viehweg's fundamental results, 
we recommend the reader to see \cite{fujino-revisited}. 

In this paper, we first prove the following 
Hodge theoretic injectivity theorem (cf.~\cite{ev}, \cite{ambro}, \cite{fujino5}, etc.). 

\begin{thm}[Fundamental injectivity theorem]\label{f-thm1.2}
Let $X$ be a complex manifold in Fujiki's class $\mathcal C$ 
and let $\Delta$ be a boundary $\mathbb R$-divisor 
on $X$ such that $\Supp \Delta$ is a simple normal crossing divisor on $X$. 
Let $\mathcal L$ be a line bundle 
on $X$ and let $D$ be an effective Weil divisor on $X$ whose support 
is contained in $\Supp \Delta$. 
Assume that $\mathcal L\sim _{\mathbb R}K_X+\Delta$. 
Then the natural homomorphism 
$$
H^q(X, \mathcal L)\to H^q(X, \mathcal L\otimes \mathcal O_X(D))
$$ 
induced by the inclusion $\mathcal O_X\to \mathcal O_X(D)$ is injective for every $q$. 
\end{thm}

It is easy to see that Theorem \ref{f-thm1.2} implies: 

\begin{thm}[Injectivity theorem]\label{f-thm1.3} Let $X$ 
be a complex manifold in Fujiki's class $\mathcal C$ and let 
$\Delta$ be a boundary 
$\mathbb R$-divisor such that 
$\Supp \Delta$ is simple normal crossing. 
Let $\mathcal L$ be a line bundle on $X$ and let $D$ be an effective 
Cartier divisor whose support contains no log canonical centers of 
$(X, \Delta)$. 
Assume the following conditions. 
\begin{itemize}
\item[(i)] $\mathcal L\sim _{\mathbb R}K_X+\Delta+H$, 
\item[(ii)] $H$ is a semi-ample $\mathbb R$-divisor, and 
\item[(iii)] $tH\sim _{\mathbb R} D+D'$ for some 
positive real number $t$, where 
$D'$ is an effective $\mathbb R$-divisor 
whose support contains no log canonical centers of $(X, \Delta)$. 
\end{itemize}
Then the homomorphisms 
$$
H^q(X, \mathcal L)\to H^q(X, \mathcal L\otimes \mathcal O_X(D))
$$ 
induced by the natural inclusion 
$\mathcal O_X\to \mathcal O_X(D)$ are injective for all $q$. 
\end{thm}

As an application of Theorem \ref{f-thm1.3}, we obtain: 

\begin{thm}[Torsion-freeness and 
vanishing theorem]\label{f-thm1.4} 
Let 
$Y$ be a complex manifold in Fujiki's class $\mathcal C$ and 
let $\Delta$ be a boundary $\mathbb R$-divisor such that 
$\Supp \Delta$ is simple normal crossing. 
Let $f:Y\to X$ be a surjective morphism onto a projective 
variety $X$ and let $\mathcal L$ be a line bundle on $Y$ such that 
$\mathcal L-(K_Y+\Delta)$ is $f$-semi-ample. 
\begin{itemize}
\item[(i)] 
Let $q$ be an arbitrary nonnegative integer. 
Then every associated prime of $R^qf_*\mathcal L$ is the generic point of 
the $f$-image of some log canonical stratum of $(Y, \Delta)$. 
\item[(ii)] 
Assume that $\mathcal L-(K_Y+\Delta)\sim _{\mathbb R}f^*H$ for 
some ample $\mathbb R$-divisor $H$ on $X$. 
Then $H^p(X, R^qf_*\mathcal L)=0$ for every $p>0$ 
and $q\geq 0$. 
\end{itemize}
\end{thm}

When $X$ and $Y$ are projective, Theorem \ref{f-thm1.3} and Theorem \ref{f-thm1.4} are 
well-known and play crucial roles in \cite{fujino2}. 

By using Theorem \ref{f-thm1.4}, we can establish: 

\begin{thm}[Semipositivity theorem]\label{f-thm1.5}
Let $X$ be a compact K\"ahler manifold and 
let $Y$ be a smooth projective variety, and let 
$f:X\to Y$ be a surjective morphism. Let $D$ be a simple normal crossing 
divisor on $X$ such that every stratum of $D$ is dominant onto $Y$. 
Let $\Sigma$ be a simple normal crossing divisor on $Y$. 
We put $Y_0=Y\setminus \Sigma$. If $f$ is smooth and 
$D$ is relatively normal crossing 
over $Y_0$, then $R^if_*\omega_{X/Y}(D)$ is the upper canonical 
extension of the bottom Hodge filtration. In particular, it is locally free. 

We further assume that all the local monodromies on the local system 
$R^{d+i}f_{0*}\mathbb C_{X_0- D_0}$ around $\Sigma$ are unipotent, then 
$R^if_*\omega_{X/Y}(D)$ is nef, where $d=\dim X-\dim Y$, 
$X_0=f^{-1}(Y_0)$, and $D_0=D|_{X_0}$.   
\end{thm}

We note that a {\em{nef}} locally free sheaf was originally 
called a ({\em{numerically}}) {\em{semipositive}} locally 
free sheaf in the literature. 
Theorem \ref{f-thm1.5} is the main ingredient of Theorem \ref{f-thm1.1}. 
In this paper, we do not use \cite[Theorem 32]{kawamata1} for 
the proof of Theorem \ref{f-thm1.1} (see Remark \ref{f-rem6.4}). 
Note that Theorem \ref{f-thm7.8} and Corollary \ref{f-cor7.11}, 
which directly follow from Theorem \ref{f-thm1.5}, are new. 

Let us discuss some applications of Theorem \ref{f-thm1.1}. 
The following conjecture is a natural formulation of Iitaka's conjecture 
for the minimal model program. 

\begin{conj}[Log Iitaka conjecture]\label{f-conj1.6}
Let $(X, \Delta)$ be a projective log canonical 
pair and let $f:X\to Y$ be a surjective morphism 
onto a normal projective  variety $Y$ with connected fibers. 
Then 
\begin{align*}
\kappa (X, K_X+\Delta)\geq \kappa (X_y, K_{X_y}+\Delta|_{X_y})+\kappa (Y)
\end{align*}
where $X_y$ is a sufficiently general fiber of $f:X\to Y$. 
Note that $\kappa (Y)$ denotes the Kodaira dimension of $Y$, that is, $\kappa (Y)
=\kappa(\widetilde Y, K_{\widetilde {Y}})$, where 
$\widetilde Y\to Y$ is a resolution of singularities. 
\end{conj}

When $\dim X=n$ and $\dim Y=m$ in Conjecture \ref{f-conj1.6}, 
we sometimes call it Conjecture $C_{n, m}^{\text{log}}$. 
If $X$ and $Y$ are smooth and $\Delta=0$, then Conjecture \ref{f-conj1.6} is nothing but 
Iitaka's original conjecture (see \cite{iitaka}). 
We can easily check that Conjecture 
\ref{f-conj1.6} holds true when $Y$ is of general type and 
$\Delta$ is a $\mathbb Q$-divisor. 
Note that Theorem \ref{f-thm1.7} below is contained in \cite{campana} 
(see also \cite[Chapter V.~4.1.~Theorem (2)]{nakayama}). 
Moreover, Campana raised the orbifold version of the Iitaka conjecture. 
For the details, see \cite[Section 4]{campana} (see also \cite{lu}). 

\begin{thm}[Addition formula]\label{f-thm1.7}
Let $(X, \Delta)$ be a projective log canonical 
pair such that $\Delta$ is a $\mathbb Q$-divisor 
and let $f:X\to Y$ be a surjective morphism 
onto a normal projective  variety $Y$ with connected fibers. 
Assume that 
$\kappa (Y)=\dim Y$. 
Then 
\begin{align*}
\kappa (X, K_X+\Delta)&= \kappa (X_y, K_{X_y}+\Delta|_{X_y})+\kappa (Y)\\
&=\kappa (X_y, K_{X_y}+\Delta|_{X_y})+\dim Y
\end{align*}
where $X_y$ is a sufficiently general fiber of $f:X\to Y$. 
\end{thm}

\begin{rem}\label{f-rem1.8}
By Nakayama (see \cite[Chapter V.~4.4.~Theorem (1)]{nakayama}), we have  
$$
\kappa_{\sigma}(X, K_X+\Delta)\geq \kappa _{\sigma}(X_y, K_{X_y}+\Delta|_{X_y})+
\kappa_{\sigma}(\widetilde Y, K_{\widetilde Y}), 
$$
where $\kappa _\sigma$ denotes Nakayama's numerical 
dimension. 
In general, it is conjectured that 
$\kappa_{\sigma}(X, K_X+\Delta)=\kappa (X, K_X+\Delta)$ 
when $\Delta$ is a $\mathbb Q$-divisor, which 
is sometimes called {\em{the generalized abundance conjecture}}. 
If $\kappa _{\sigma}(X, K_X+\Delta)=\kappa(X, K_X+\Delta)$, then 
we have 
\begin{align*}
\kappa (X, K_X+\Delta)&\geq \kappa_{\sigma}(X_y, K_{X_y}+\Delta|_{X_y})
+\kappa_{\sigma}( \widetilde {Y}, K_{\widetilde Y})
\\ & \geq \kappa (X_y, K_{X_y}+\Delta|_{X_y})+\kappa (\widetilde Y, K_{\widetilde Y}). 
\end{align*}
Therefore, Conjecture \ref{f-conj1.6} follows from 
the generalized abundance conjecture when $\Delta$ is a $\mathbb Q$-divisor. 
\end{rem}

Theorem \ref{f-thm1.9} is due to Maehara (see \cite[Corollary 2]{maehara}). 
In this paper, we recover it as an application of Theorem \ref{f-thm1.1}. 

\begin{thm}[Addition formula for logarithmic Kodaira 
dimensions]\label{f-thm1.9} 
Let $f:X\to Y$ be a surjective morphism between smooth projective varieties 
with connected fibers. 
Let $D_X$ {\em{(}}resp.~$D_Y${\em{)}} be a simple normal crossing divisor 
on $X$ {\em{(}}resp.~$Y${\em{)}}. 
Assume that $\Supp f^*D_Y\subset \Supp D_X$. 
We further assume that $\kappa (Y, K_Y+D_Y)=\dim Y$. 
Then we have 
\begin{align*}
\kappa (X, K_X+D_X)&=\kappa (F, K_F+D_X|_F)+\kappa (Y, K_Y+D_Y)\\
&=\kappa (F, K_F+D_X|_F)+\dim Y, 
\end{align*} 
where $F$ is a sufficiently general fiber of $f:X\to Y$. 

We put $X^0=X\setminus D_X$, $Y^0=Y\setminus D_Y$, and $F^0=F|_{X^0}$. 
Then the above equality is nothing but 
\begin{align*}
\overline {\kappa}(X^0)&=\overline{\kappa}(F^0)+\overline{\kappa}(Y^0)\\
&=\overline{\kappa}(F^0)+\dim Y^0. 
\end{align*}
Note that $\overline {\kappa}$ denotes Iitaka's logarithmic 
Kodaira dimension {\em{(}}see \cite{iitaka2}{\em{)}}. 
\end{thm}

We will quickly prove Theorem \ref{f-thm1.7} and Theorem \ref{f-thm1.9} 
in Section \ref{f-sec9} and Section 
\ref{f-sec10} respectively by using Theorem \ref{f-thm1.1} and \cite{ak}. 
In \cite{fujino-sub}, we prove the subadditivity of the logarithmic Kodaira 
dimension for affine varieties. 

We summarize the contents of this paper. 
Section \ref{f-sec2} collects some basic results and definitions. 
In Section \ref{f-sec3}, we prove the fundamental injectivity 
theorem:~Theorem \ref{f-thm1.2}. 
The proof of Theorem \ref{f-thm1.2} uses the theory of mixed Hodge 
structures. 
Section \ref{f-sec4} is devoted to the theory of mixed Hodge 
structures for cohomology with compact support. 
In Section \ref{f-sec5}, we prove Theorem \ref{f-thm1.3} 
and Theorem \ref{f-thm1.4}. These are direct consequences of Theorem \ref{f-thm1.2}. 
In Section \ref{f-sec6}, we explain the semipositivity theorem:~Theorem \ref{f-thm1.5}. 
In Section \ref{f-sec7}, 
we discuss weakly positive sheaves. 
Section \ref{f-sec8} is the main part of this paper. 
It is devoted to the proof of the twisted weak positivity theorem:~Theorem \ref{f-thm1.1}. 
We prove Theorem \ref{f-thm1.7} (resp.~Theorem \ref{f-thm1.9}) 
in Section \ref{f-sec9} (resp.~Section \ref{f-sec10}) as an application of 
Theorem \ref{f-thm1.1}. 

In this paper, we discuss neither Nakayama's sophisticated 
treatment of weak positivity in \cite[Chapter V.~\S 3]{nakayama} 
nor Schnell's results on weak positivity coming from 
Saito's theory of mixed Hodge modules (see \cite{schnell}). 
We naively discuss some generalizations of Viehweg's weak positivity 
following \cite{viehweg1}, \cite{campana}, etc. 
The main motivation of this paper is to understand and clarify 
Viehweg's clever covering arguments 
used for the proof of his famous weak positivity theorem 
(see also \cite{fujino-revisited}). 

\begin{ack}
The author was partially supported 
by Grant-in-Aid for Young Scientists (A) 24684002 from JSPS. 
He would like to thank Professors Akira Fujiki and 
Kazuhisa Maehara for answering his questions. 
He also would like to thank Professor Noboru Nakayama 
for useful comments and discussions. 
He thanks the referees for useful comments. 
Finally, he thanks Yoshinori Gongyo for pointing out a typo. 
\end{ack}

We will use the standard notation of the minimal model 
program as in \cite{fujino2}. 
In this paper, we always assume that complex varieties are Hausdorff 
and countable at infinity. 
For the basic theory of complex varieties, see, for example, \cite{bs}, \cite{fischer}, and 
\cite{nakayama}. The style of this paper 
is the same as that of \cite{fujino2} (see \cite{fujino-book}, \cite{fujino0}, 
\cite{fujino-fujisawa}, \cite{fujino-foundation}, \cite{fujino-zucker65}, 
etc.). 
Our results depend on the theory of variations of mixed 
Hodge structure 
(see \cite{fujino1}, \cite{fujino-fujisawa}, and \cite{ffs}). 

\section{Preliminaries}\label{f-sec2} 

Let us start with some remarks on {\em{canonical divisors}}. 

\begin{say}[Canonical divisors]\label{f-say2.1} 
We consider complex variety $X$, which 
is not necessarily algebraic. 

\begin{rem}\label{f-rem2.2} 
(i) Let $\omega^{\bullet}_X$ be the {\em{dualizing complex}} 
of a complex variety $X$ (see, for example, \cite{rr}, \cite{rrv}, and \cite{bs}). 
We put $\omega_X=\mathcal H^{-d}(\omega^\bullet_X)$, where 
$d=\dim X$, and call it the {\em{canonical sheaf}} of $X$. 
When $X$ is a compact complex manifold, it is well-known that 
$\omega_X\simeq \Omega^{d}_X$. 
For the details of $\omega^\bullet _X$, see, for example, \cite[Chapter VII \S2]{bs}. 

(iii) Some complex variety $X$ does not 
admit any nonzero meromorphic section of $\omega_X$. 
However, if there is no risk of confusion, 
we use the symbol $K_X$ as a formal 
divisor class with an isomorphism $\mathcal O_X(K_X)\simeq \omega_X$ and 
call it the {\em{canonical divisor}} of $X$. See \cite[Chapter II.~\S 4]{nakayama}. 
\end{rem}

\begin{rem}\label{f-rem2.3}
Let $D$ be a Cartier divisor and let $\mathcal L$ be a 
line bundle on a complex variety $X$. 
If there is no risk of confusion, we sometimes write 
$$
\mathcal O_X(K_X+D+\mathcal L)
$$
in order to express 
$$
\omega_X\otimes \mathcal O_X(D)\otimes \mathcal L. 
$$
For simplicity, we sometimes use $\mathcal L^N$ to denote $\mathcal L^{\otimes N}$ 
if there is no risk of confusion. 
\end{rem}
\end{say}

In this paper, all complex varieties are {\em{algebraic}} or {\em{compact}}. 
Therefore, there are no subtle problems in the following definitions.  

\begin{say}[Singularities of pairs]\label{f-say2.4} 
Let us recall the definition of singularities of pairs. 

Let $X$ be a normal variety and let $\Delta$ be an effective $\mathbb R$-divisor 
on $X$ such that $K_X+\Delta$ is $\mathbb R$-Cartier. 
Let $f:Y\to X$ be a resolution such that 
$\Exc (f)\cup f^{-1}_*\Delta$ has a simple normal crossing 
support, where $\Exc(f)$ is the 
exceptional locus of $f$ and 
$f^{-1}_*\Delta$ is the strict transform of $\Delta$ on $Y$. 
We write $$K_Y=f^*(K_X+\Delta)+\sum _i a_i E_i$$ and 
$a(E_i, X, \Delta)=a_i$. 
We say that $(X, \Delta)$ is {\em{lc}} if and only if 
$a_i\geq -1$ for every $i$. 
Note that the {\em{discrepancy}} $a(E, X, \Delta)\in \mathbb R$ can be 
defined for every prime divisor $E$ {\em{over}} $X$. 
It is well-known that $(X, \Delta)$ is lc if and only if 
$a(E, X, \Delta)\geq -1$ for every 
prime divisor $E$ over $X$. 
Let $(X, \Delta)$ be an lc pair. If there is a resolution $f:Y\to X$ such that 
$\Exc(f)$ is a divisor, $\Exc(f)\cup f^{-1}_*\Delta$ has 
a simple normal crossing support, 
and $a(E, X, \Delta)>-1$ for every $f$-exceptional 
divisor $E$, then $(X, \Delta)$ is called {\em{dlt}}.  
Here, lc (resp.~dlt) is an abbreviation of 
{\em{log canonical}} (resp.~{\em{divisorial log terminal}}). 

For the details and various examples of 
singularities of pairs, see, for example, \cite{fujino-what} 
(see also \cite[Section 2.3]{fujino-foundation}). 

\begin{rem}[Szab\'o's resolution lemma]\label{f-rem2.5} 
We note that Szab\'o's resolution lemma (see, for example, \cite[3.5 Resolution 
lemma]{fujino-what}) now holds for compact complex varieties. 
For the details, see, for example, \cite[Theorem 10.45, Proposition 10.49, and 
the proof of (10.45)]{kollar-book}. 
We will use Szab\'o's resolution lemma repeatedly in this paper. 
\end{rem}

Let us recall the definition of {\em{log canonical centers}}. 

\begin{defn}[Log canonical center]\label{f-def2.6} 
Let $(X, \Delta)$ be a log canonical pair. 
If there is a resolution $f:Y\to X$ and a prime divisor $E$ on 
$Y$ such that $a(E, X, \Delta)=-1$, then $f(E)$ is 
called a {\em{log canonical center}} of $(X, \Delta)$. 
\end{defn} 

Definition \ref{f-def2.7} is useful for torsion-free theorem. 

\begin{defn}[Log canonical stratum]\label{f-def2.7}  
Let $(X, \Delta)$ be a log canonical pair. 
A {\em{log canonical stratum}} (an {\em{lc stratum}}, for 
short) of $(X, \Delta)$ is $X$ itself or a 
log canonical center of $(X, \Delta)$. 
Note that $X$ is a log canonical stratum of $(X, \Delta)$ but 
is not a log canonical center of $(X, \Delta)$. 
\end{defn}

\begin{say}[Divisors] 
Let us recall some basic operations for $\mathbb Q$-divisors and 
$\mathbb R$-divisors. 

For an $\mathbb R$-divisor $D=\sum _{i=1}^rd_i D_i$ such that 
$D_i$ is a prime divisor for every $i$ and $D_i\ne D_j$ for 
$i\ne j$, we define 
the {\em{round-down}} $\lfloor D\rfloor=\sum _{i=1}^r\lfloor d_i \rfloor 
D_i$ 
(resp.~the {\em{round-up}} $\lceil D\rceil =\sum _{i=1}^r\lceil d_i \rceil D_i$), 
where for every real number $x$, $\lfloor x\rfloor$ (resp.~$\lceil x\rceil$) is the integer 
defined by 
$x-1<\lfloor x\rfloor \leq x$ (resp.~$x\leq \lceil x\rceil <x+1$). 
The {\em{fractional part}} $\{D\}$ of $D$ denotes 
$D-\lfloor D\rfloor$. 
We also define $D^{=1}=\sum _{d_i=1}D_i$. 
We call $D$ a {\em{boundary}} $\mathbb R$-divisor if $0\leq d_i\leq 1$ for 
every $i$. 

\begin{rem}\label{f-rem2.9}
Let $X$ be a compact complex manifold and let $D_1$, $D_2$, $\cdots$, $D_k$ be 
Cartier divisors on $X$. 
We consider the linear map 
$$
\varphi: \mathbb R^k\longrightarrow \Pic (X)\otimes \mathbb R
$$ 
defined by $\varphi(r_1, r_2, \cdots, r_k)=r_1D_1+r_2D_2+\cdots +r_k D_k$, 
which is defined over $\mathbb Q$. 
Let $\mathcal L$ be a line bundle on $X$. 
Then $\mathcal L\sim _{\mathbb R}\sum _{i=1}^{k}r_iD_i$ means 
$\mathcal L=\varphi(r_1, r_2, \cdots, r_k)$ in $\Pic (X)\otimes 
\mathbb R$. 
Note that $\varphi^{-1}(\mathcal L)$ is an affine subspace of 
$\mathbb R^k$ defined over $\mathbb Q$. 
Therefore, we can find $(r'_1, r'_2, \cdots, r'_k)\in \mathbb Q^k$ such that 
$\mathcal L\sim _{\mathbb Q}\sum _{i=1}^{k}r'_i D_i$, that is, 
$\mathcal L=\varphi (r'_1, r'_2, \cdots, r'_k)$ in $\Pic (X)\otimes \mathbb Q$ 
if $\varphi^{-1}(\mathcal L)$ is not empty. 
\end{rem}
\end{say}
\end{say}

\begin{say}[Fujiki's class $\mathcal C$] 

In this paper, we use the notion of complex varieties in Fujiki's class $\mathcal C$. 

\begin{defn}[Fujiki's class $\mathcal C$]
Let $X$ be a compact reduced complex analytic space. Then $X$ is {\em{in Fujiki's class 
$\mathcal C$}} if and only if there is a surjective morphism 
$f:Y\to X$ with $Y$ a compact K\"ahler manifold. 
It is well-known that $X$ is in Fujiki's class $\mathcal C$ if and only if 
there is a bimeromorphic morphism $g:V\to X$ from a compact K\"ahler manifold $V$ 
(see, for example, \cite[Th\'eor\`eme 3]{var}).  
\end{defn}

It is well-known that some basic results 
on the minimal model program can be generalized for 
varieties in Fujiki's class $\mathcal C$. See \cite{nakayama1}, 
\cite[Section 4]{fujino-kawa}, etc. 

\begin{rem}\label{f-rem2.12}
For the details of complex varieties in Fujiki's class $\mathcal C$, 
(locally) K\"ahler morphisms, and so on, see \cite{fujiki}, \cite{fujiki2}, and \cite{var}. 
Note that every (locally) projective morphism is (locally) 
K\"ahler and that the composition of two locally K\"ahler 
morphisms is again locally K\"ahler 
(see \cite[(1.2), (2.1), (2.2), and so on]{fujiki2}). 
\end{rem}
\end{say}

\begin{say}[Simple normal crossing varieties] 
In Section \ref{f-sec4}, we will use the {\em{Mayer--Vietoris simplicial resolution}} 
of a simple normal crossing variety $X$ in order to discuss various 
mixed Hodge structures. 

\begin{defn}[Mayer--Vietoris simplicial resolution]\label{f-def2.14}
Let $X$ be a simple normal crossing variety with 
the irreducible decomposition $X=\bigcup_{i\in I}X_i$. 
Let $I_n$ be the set of strictly increasing 
sequences $(i_0, \cdots, i_n)$ in $I$ and $X^n=\coprod_{I_n} X_{i_0}\cap 
\cdots \cap X_{i_n}$ the disjoint union of the intersections of 
$X_i$. 
Let $\varepsilon_n:X^n \to X$ be the disjoint union 
of the natural inclusions. 
Then $\{X^n, \varepsilon_n\}_n$ 
has a natural semi-simplicial structure. 
The face operator is induced by $\lambda_{j,n}$, where 
$$\lambda_{j, n}: X_{i_0}\cap \cdots 
\cap X_{i_n}\to 
X_{i_0}\cap \cdots \cap X_{i_{j-1}}\cap X_{i_{j+1}}\cap 
\cdots \cap X_{i_n}$$ is the natural closed embedding for 
$j\leq n$ (cf.~\cite[3.5.5]{elzein2}). 
We denote it by $\varepsilon:X^{\bullet}\to X$ and call it the 
{\em{Mayer--Vietoris simplicial resolution}} of 
$X$. The complex 
$$
0\to \varepsilon_{0*}\mathcal O_{X^0}
\to \varepsilon _{1*}\mathcal O_{X^1} \to 
\cdots \to \varepsilon _{k*}\mathcal O_{X^k}\to \cdots,  
$$ 
where the differential 
$d_k:\varepsilon _{k*}\mathcal O_{X^k}\to 
\varepsilon _{k+1*}\mathcal O_{X^{k+1}}$ is 
$\sum ^{k+1}_{j=0} (-1)^j 
\lambda^* _{j, k+1}$ for every $k\geq 0$, 
is denoted by $\mathcal O_{X^\bullet}$.  
We see that $\mathcal O_{X^\bullet}$ is 
quasi-isomorphic to $\mathcal O_X$. 
By tensoring $\mathcal L$, any line bundle 
on $X$, 
to $\mathcal O_{X^\bullet}$, we obtain a complex 
$$
0\to \varepsilon_{0*}\mathcal L^0
\to \varepsilon _{1*}\mathcal L^1 \to 
\cdots \to \varepsilon _{k*}\mathcal L^k\to \cdots,  
$$ 
where $\mathcal L^n=\varepsilon ^*_n\mathcal L$. 
Here, $\mathcal L^n$ does not mean $\mathcal L^{\otimes n}$ 
(see Remark \ref{f-rem2.3}). 
It is denoted by $\mathcal L^{\bullet}$. 
Of course, $\mathcal L^{\bullet}$ is quasi-isomorphic to 
$\mathcal L$. 
We note that 
$\mathbb H^q(X, \mathcal L^{\bullet})$ 
is obviously isomorphic to 
$H^q(X, \mathcal L)$ for every $q\geq 0$ 
because $\mathcal L^\bullet$ is quasi-isomorphic to $\mathcal L$. 

We note that a {\em{stratum}} 
of $X$ means an irreducible component of $X_{i_0}\cap \cdots \cap X_{i_k}$ 
for some $\{i_0, \cdots, i_k\}\subset I$. 
If $X$ is a simple normal crossing divisor on a smooth 
variety $M$, then a stratum of $X$ is nothing but a log canonical 
center of $(M, X)$. 
\end{defn}
\end{say}

\begin{say}[Flat base change theorem]\label{f-say2.15}
In the proof of Theorem \ref{f-thm1.9}, we will use the flat base change 
theorem for relative dualizing sheaves (see \cite[\S 3]{viehweg1} and 
\cite[Section 4]{mori}). We need the following statement. 

\begin{thm}\label{f-rem2.16} 
Let $f:V\to W$ be a flat projective surjective morphism 
from a Cohen--Macaulay quasi-projective 
variety $V$ to a smooth quasi-projective variety $W$. 
Let $g:W'\to W$ be a finite 
flat morphism 
from a smooth quasi-projective variety $W'$. 
We consider the following diagram: 
$$
\xymatrix{
V'\ar[r]^{h} \ar[d]_{f'} & V\ar[d]^{f} \\
   W' \ar[r]_{g} & W
} 
$$
where $V'=W'\times _W V$. 
Then 
we have 
$$h^*\omega_{V/W}=\omega_{V'/W'}.$$
Note that 
$$
\omega_{V/W}=\omega_V\otimes f^*\omega_W^{\otimes -1} 
\quad \text{and}
\quad 
\omega_{V'/W'}=\omega_{V'}\otimes f'^*
\omega_{W'}^{\otimes -1}. 
$$
\end{thm}

Theorem \ref{f-rem2.16} is a very special case of the 
flat base change theorem (see \cite[Theorem 2]{verdier}). 
See also \cite{hartshorne}, \cite{conrad}, etc. 
The author does not know if the flat base change theorem 
(\cite[Theorem 2]{verdier}) is true or not 
in the analytic category (cf.~\cite{rr} and \cite{rrv}). 
Therefore, 
we do not use the flat base change theorem in the proof of 
Theorem \ref{f-thm1.1} (see \cite[Lemma 3.2]{viehweg1} and 
\cite[(4.10) Base change theorem]{mori}). 
Note that $X$ in Theorem \ref{f-thm1.1} is not necessarily {\em{algebraic}}. 
\end{say} 

\begin{say}[Relative vanishing theorems]\label{f-say2.17} 
The following theorem is a relative version of the Kawamata--Viehweg 
vanishing theorem for generically finite morphisms. 

\begin{thm}[{cf.~\cite[Theorem 3.6]{nakayama1}}]\label{f-thm2.18} 
Let $f:X\to Y$ be a proper generically finite 
morphism from a compact complex manifold $X$ onto a complex variety $Y$ and 
let $\Delta$ be a $\mathbb Q$-divisor on $X$ such that 
$\Supp \Delta$ is a simple normal crossing divisor and $\lfloor \Delta\rfloor=0$. 
Let $\mathcal L$ be a line bundle on $X$. 
Assume that $\mathcal L-(K_X+\Delta)$ 
is $f$-nef. 
Then $R^if_*\mathcal L=0$ for every $i>0$. 
\end{thm}

Theorem \ref{f-thm2.18} is a special case of \cite[Corollary 1.3]{fujino3}. 
For the details, see \cite{nakayama1}, \cite{fujino3}, etc. 
Lemma \ref{f-lem2.19}, which is an easy consequence of Theorem \ref{f-thm2.18}, 
is very useful and indispensable. 

\begin{lem}[{Reid--Fukuda type (see \cite[Lemma]{fukuda})}]
\label{f-lem2.19} 
Let $X$ be a compact complex manifold and let $\Delta$ be a boundary 
$\mathbb Q$-divisor on $X$ such that $\Supp \Delta$ is a simple normal 
crossing divisor on $X$. Let $f:X\to Y$ be a bimeromorphic 
morphism onto a compact complex variety $Y$. 
Assume that $f$ is an isomorphism at the general points of any 
log canonical center of $(X, \Delta)$ 
and that $\mathcal L$ is a line bundle 
on $X$ such that $\mathcal L-(K_X+\Delta)$ is $f$-nef. 
Then $R^if_*\mathcal L=0$ for every $i>0$. 
\end{lem}

\begin{proof}
By using induction on the number of irreducible components 
of $\lfloor \Delta\rfloor$ and on the dimension of 
$X$, 
we can quickly prove Lemma \ref{f-lem2.19} by Theorem \ref{f-thm2.18}. 
For the details, see, for example, the proof of \cite[Lemma 6.2]{fujino2}. 
\end{proof}

We close this section with a remark on the relative Kawamata--Viehweg vanishing 
theorem. Anyway, the proof of Theorem \ref{f-thm2.18} when $Y$ is not 
algebraic is much harder than the case when $Y$ is algebraic. 

\begin{rem}[Projective versus K\"ahler]
We are mainly interested in {\em{projective}} varieties. 
This is because the minimal model program works well only for {\em{projective}} 
varieties. However, in this paper, we treat K\"ahler manifolds and  
complex varieties in Fujiki's class $\mathcal C$ in order to 
cover Campana's result (see \cite[Theorem 4.13]{campana}, which is 
essentially equivalent 
to Theorem \ref{f-thm1.1}). 
If the reader is only interested in {\em{projective}} varieties, 
then we recommend the reader to read this paper assuming 
that all the varieties are {\em{projective}}. 
For the minimal model program for compact K\"ahler manifolds and 
some related topics, see \cite{chp}, \cite{fujino-someMMP}, 
\cite{hp1}, \cite{hp2}, etc. 

Let $f:X\to Y$ be a projective bimeromorphic morphism 
from a compact complex manifold $X$ to a compact K\"ahler manifold $Y$. 
Let $D$ be an $f$-nef Cartier divisor on $X$ such that 
the support of the fractional part $\{D\}$ of $D$ is a simple 
normal crossing divisor on $X$. 
Then $R^if_*\mathcal O_X(K_X+\lceil D
\rceil)=0$ for 
every $i>0$ by Theorem \ref{f-thm2.18}. 

If $Y$ is {\em{projective}}, then the above vanishing easily follows 
from the usual 
Kawamata--Viehweg vanishing theorem for {\em{projective}} varieties 
(see \cite[Proposition 2.69]{kollar-mori} and 
the proof of Proposition \ref{f-prop7.15} below). 
This means that the relative vanishing theorem 
follows from the vanishing theorem for projective varieties. 
On the other hand, if $Y$ is K\"ahler but not projective, 
then the above vanishing theorem is much harder to prove. 
\end{rem}
\end{say}

\section{Fundamental injectivity theorem}\label{f-sec3} 
In this section, we prove Theorem \ref{f-thm1.2}. 
Theorem \ref{f-thm1.2} is a direct consequence of 
the $E_1$-degeneration of Hodge to de Rham spectral sequence 
associated to the 
mixed Hodge structure for cohomology with compact support. 
We discuss the $E_1$-degeneration in Section \ref{f-sec4}. 

\begin{proof}[Proof of Theorem \ref{f-thm1.2}]
Without loss of generality, we may assume 
that $X$ is connected. We put $S=\lfloor \Delta\rfloor$ and $B=\{\Delta\}$. 
By perturbing $B$, 
we may assume that $B$ is a $\mathbb Q$-divisor (see Remark \ref{f-rem2.9}).  
We put $\mathcal M=\mathcal O_X(\mathcal L-K_X-S)$. 
Let $N$ be the smallest positive integer 
such that $\mathcal L^N\sim N(K_X+ S+ B)$. 
In particular, $N B$ is an integral Weil divisor. 
We take the $N$-fold cyclic cover 
$$\pi': Y'=\Spec \bigoplus _{i=0}^{N-1} \mathcal M^{-i}\to 
X$$ associated to 
the section $N B\in |\mathcal M^{N}|$. 
More precisely, let $s\in H^0(X, \mathcal M^{N})$ be a section 
whose zero divisor is $N B$. 
Then the dual of $s:\mathcal O_X\to \mathcal M^{N}$ 
defines an $\mathcal O_X$-algebra structure on 
$\bigoplus ^{N-1}_{i=0} \mathcal M^{-i}$. 
Let $Y\to Y'$ be the normalization and 
let $\pi:Y\to X$ be the composition morphism. 
It is well-known that 
$$
Y=\Spec\bigoplus _{i=0}^{N-1}\mathcal M^{-i}(\lfloor iB\rfloor). $$  
For the details, see \cite[3.5.~Cyclic covers]{ev}. 
Note that $Y$ has only quotient singularities. 
We put $T=\pi^*S$. 
We note that 
$T$ is Cartier. 
Hence the locally free sheaf $\mathcal O_Y(-T)$ is the defining ideal sheaf of $T$ on $Y$. 
The $E_1$-degeneration of 
\begin{align*}\tag{$\clubsuit$}\label{f-shiki1}
E^{p, q}_1=H^q(Y, 
\widetilde {\Omega}^p_Y(\log T)(-T))
\Rightarrow H^{p+q}(Y, j_!\mathbb C_{Y-T}), 
\end{align*} 
where $j:Y-T\to Y$ is the natural open immersion, 
implies that the homomorphism 
$$
H^q(Y, j_!\mathbb C_{Y-T})\to H^q(Y, \mathcal O_Y(-T))
$$
induced by the natural inclusion 
$$
j_!\mathbb C_{Y-T}\subset \mathcal O_Y(-T)
$$ 
is surjective for every $q$. 
For the definition of $\widetilde {\Omega}^p_Y(\log T)(-T)$, 
see Definition \ref{f-def4.5}. 
We will discuss the $E_1$-degeneration of (\ref{f-shiki1}) in Section \ref{f-sec4} 
below. 
By taking a suitable direct summand $$\mathcal C\subset \mathcal M^{-1}(-S)$$ of 
$$\pi_*(j_!\mathbb C_{Y-T})\subset \pi_*\mathcal O_Y(-T), $$ we 
obtain a surjection 
$$
H^q(X, \mathcal C)\to H^q(X, \mathcal M^{-1}(-S)) 
$$ 
induced by the natural inclusion $\mathcal C\subset \mathcal M^{-1}(-S)$ for 
every $q$. We can check 
the following simple property by examining the monodromy 
action of the Galois group $\mathbb Z/N\mathbb Z$ of 
$\pi:Y\to X$ 
on $\mathcal C$ around $\Supp B$ (see also the proof of 
\cite[2.12.1 Proposition]{kollar-sing}). 

\begin{lem}[{cf.~\cite[Corollary 2.54]{kollar-mori}}]\label{f-lem3.1}
Let $U\subset X$ be 
a connected open subset such that 
$U\cap \Supp \Delta\ne \emptyset$. 
Then $H^0(U, \mathcal C|_{U})=0$. 
\end{lem}
\begin{proof}
If $U\cap \Supp B\ne \emptyset$, then $H^0(U, \mathcal C|_U)=0$ since the monodromy 
action on $\mathcal C|_{U\setminus 
\Supp B}$ around $\Supp B$ is nontrivial. 
If $U\cap \Supp S\ne \emptyset$, then $H^0(U, \mathcal C|_U)=0$ since 
$\mathcal C$ is a direct summand of $\pi_*(j_!\mathbb C _{Y-T})$ and $T=\pi^*S$. 
\end{proof}

Lemma \ref{f-lem3.1} is useful by the following fact. 
The proof of Lemma \ref{f-lem3.2} is obvious. 

\begin{lem}[{cf.~\cite[Lemma 2.55]{kollar-mori}}]\label{f-lem3.2}
Let $F$ be a sheaf of Abelian groups on a topological 
space 
$X$ and let $F_1$ and $F_2$ be subsheaves of $F$. 
Let $Z\subset X$ be a closed subset. Assume 
that 
\begin{itemize}
\item[(1)] $F_2|_{X- Z}=F|_{X- Z}$, and 
\item[(2)] if $U\subset X$ is a connected open 
subset with $U\cap Z\ne \emptyset$, then 
$H^0(U, F_1|U)=0$. 
\end{itemize}
Then $F_1$ is a subsheaf of $F_2$. 
\end{lem}

Therefore, we obtain: 

\begin{cor}[{cf.~\cite[Corollary 2.56]{kollar-mori}}]\label{f-cor3.3}
Let $M\subset \mathcal M^{-1}(-S)$ be a subsheaf such that 
$M|_{X- \Supp \Delta}
=\mathcal M^{-1}(-S)|_{X- \Supp \Delta}$. 
Then the injection 
$$
\mathcal C\to \mathcal M^{-1}(-S) 
$$ 
factors as 
$$
\mathcal C \to M\to \mathcal M^{-1}(-S). 
$$ Therefore, 
$$H^q(X, M)\to H^q(X, \mathcal M^{-1}(-S))
$$ 
is surjective for every $q$. 
\end{cor}

\begin{proof}
The first part is clear from Lemma \ref{f-lem3.1} 
and Lemma \ref{f-lem3.2}. 
This implies that we have maps 
$$
H^q(X, \mathcal C)\to H^q(X, M)\to H^q(X, \mathcal M^{-1}(-S)). 
$$ 
As we saw above, the composition is surjective. 
Hence so is the map on the right. 
\end{proof}

Therefore, 
$H^q(X, \mathcal M^{-1}(-S-D))\to H^q(X, 
\mathcal M^{-1}(-S))$ is 
surjective for every $q$. By Serre duality, we obtain that 
$$H^q(X, \mathcal O_X(K_X)\otimes \mathcal M(S))\to H^q(X, \mathcal 
O_X(K_X)\otimes \mathcal M(S+D))$$ is 
injective for every $q$. 
This means that $$H^q(X, \mathcal L)\to 
H^q(X, \mathcal L\otimes \mathcal O_X(D))$$ is injective for every $q$.
\end{proof}

\section{MHS for cohomology with compact support}\label{f-sec4}
In this section, we prove the $E_1$-degeneration of (\ref{f-shiki1}) in the proof of 
Theorem \ref{f-thm1.2} for the reader's convenience. 
It is more or less well-known to the experts.  

From \ref{f-say4.1} to \ref{f-say4.3}, we recall some well-known 
results on mixed Hodge structures. 
We use the notations in \cite{deligne} freely. 
The basic references on this topic are 
\cite[Section 8]{deligne}, \cite[Part II]{elzein}, 
\cite[Chapitres 2 and 3]{elzein2}, and the book \cite{ps}. 

First, we start with the pure Hodge structures on 
complex manifolds in Fujiki's class $\mathcal C$.

\begin{say}\label{f-say4.1} 
Let $X$ be a complex manifold in Fujiki's class $\mathcal C$. 
Then the triple $(\mathbb Z_X, 
(\Omega^{\bullet}_X, F), \alpha)$, 
where $\Omega^{\bullet}_X$ is the holomorphic de Rham complex 
with the filtration b\^ete $F$ (see 
\cite[(1.4.7)]{deligne}) and 
$\alpha:\mathbb C_X\to \Omega^{\bullet}_X$ is the inclusion, 
is a cohomological Hodge complex (CHC, for short) of weight 
zero. 

If we define weight filtrations as follows: 
\begin{align*}
W_m\mathbb Q_X=\begin{cases}0 &\text{if}\quad m<0\\ 
\mathbb Q_X &\text{if} \quad m\geq 0
\end{cases} 
\end{align*}
and 
\begin{align*}
W_m\Omega^\bullet_X=\begin{cases}0 &\text{if}\quad m<0\\ 
\Omega^\bullet_X &\text{if} \quad m\geq 0 ,
\end{cases}
\end{align*}
then we can see that 
$$(\mathbb Z_X, (\mathbb Q_X, W), (\Omega^\bullet_X, F, W))$$ is 
a cohomological mixed Hodge complex (CMHC, for short). 
We need these weight filtrations in the following arguments. 
\end{say}
The next one is also a fundamental 
example. For the details, see \cite[I.1]{elzein} or 
\cite[3.5]{elzein2}. 

\begin{say}\label{f-say4.2}
Let $D$ be a simple normal crossing  
variety in Fujiki's class $\mathcal C$. 
Let $\varepsilon:D^{\bullet}\to D$ be the Mayer--Vietoris simplicial 
resolution (see Definition \ref{f-def2.14}). 
We use similar notations to those in Definition \ref{f-def2.14}.  
The following complex of sheaves, denoted by 
$\mathbb Q_{D^{\bullet}}$, 
$$
{\varepsilon_0}_*\mathbb Q_{D^0}
\to {\varepsilon_1}_*\mathbb Q_{D^1}\to \cdots 
\to {\varepsilon_k}_*\mathbb Q_{D^k}\to \cdots, 
$$ is 
a resolution of $\mathbb Q_D$. 
More explicitly, the differential 
$d_k: \varepsilon _{k*} \mathbb Q_{D^k} 
\to \varepsilon _{k+1*}\mathbb Q_{D^{k+1}}$ is $\sum _{j=0}^{k+1} 
(-1)^j\lambda^*_{j, k+1}$ for every $k\geq 0$. 
The weight filtration $W$ on $\mathbb Q_{D^{\bullet}}$ is defined by 
\begin{align*}
W_{-q}(\mathbb Q_{D^{\bullet}})
=(0\to \cdots \to 0\to\varepsilon _{q*}\mathbb Q_{D^q}\to \varepsilon _{q+1*}\mathbb Q_{D^{q+1}}\to \cdots). 
\end{align*}
We obtain the resolution $\Omega^{\bullet}_{D^{\bullet}}$ of 
$\mathbb C_D$ as follows: 
$$
{\varepsilon_0}_*\Omega^{\bullet}_{D^0}
\to {\varepsilon_1}_*\Omega^{\bullet}_{D^1}\to \cdots 
\to {\varepsilon_k}_*\Omega^{\bullet}_{D^k}\to \cdots. 
$$ 
Of course, 
$d_k: \varepsilon _{k*} \Omega^{\bullet}_{D^k} 
\to \varepsilon _{k+1*}\Omega^{\bullet}_{D^{k+1}}$ is $\sum _{j=0}^{k+1} 
(-1)^j\lambda^*_{j, k+1}$. 
Let $s(\Omega^{\bullet}_{D^{\bullet}})$ be the single complex 
associated to the double complex $\Omega^{\bullet}_{D^{\bullet}}$. 
The Hodge filtration $F$ on 
$s(\Omega^{\bullet}_{D^{\bullet}})$ is defined 
by  
$$F^p=s(0\to \cdots \to 0\to 
\varepsilon _*\Omega^p_{D^{\bullet}}\to 
\varepsilon_*\Omega^{p+1}_{D^{\bullet}}\to \cdots).$$ 
We note that $$\varepsilon _*\Omega^p_{D^{\bullet}}
=(\varepsilon _{0*}\Omega^p_{D^0}
\to \varepsilon _{1*}\Omega^p_{D^1}\to \cdots 
\to \varepsilon _{k*}\Omega^p_{D^k}\to \cdots) $$
for every $p$.  
The weight filtration $W$ on $s(\Omega^{\bullet}_{D^{\bullet}})$ is defined 
by 
\begin{align*}
W_{-q}(s(\Omega^{\bullet}_{D^{\bullet}}))
=s(0\to \cdots \to 0 \to \varepsilon _{q*}\Omega^{\bullet}_{D^q}
\to \varepsilon _{q+1*}\Omega^{\bullet}_{D^{q+1}}\to \cdots). 
\end{align*}
We note that 
$$
\Gr _{-q}^W\mathbb Q_{D^\bullet}\simeq \varepsilon _{q*}\mathbb Q_{D^q}[-q] 
$$ 
and 
$$
\Gr_{-q}^W(s(\Omega^{\bullet}_{D^\bullet}))\simeq \varepsilon _{q*}
\Omega^{\bullet}_{D^q}[-q]. 
$$ 
Then $$(\mathbb Z_{D}, (\mathbb Q_{D^{\bullet}}, W), 
(s(\Omega^{\bullet}_{D^{\bullet}}), W, F))$$ is a CMHC. Here, we 
omitted the quasi-isomorphisms $\alpha: \mathbb Z_D\otimes 
\mathbb Q\to \mathbb Q_{D^\bullet}$ and 
$\beta: (\mathbb Q_{D^\bullet}, W)\otimes \mathbb C\to (s(\Omega^\bullet_{D^\bullet}), W)$ since 
there is no risk of confusion. 
This CMHC induces a natural mixed Hodge structure on $H^{\bullet}
(D, \mathbb Z)$. 
We note that the spectral sequence with respect to $W$ on $\mathbb Q_{D^\bullet}$ 
is 
\begin{align*}
{}_{W}\!E^{p, q}_1=H^{p+q}(D, \Gr_{-p}^W \mathbb Q_{D^\bullet})&=
H^{p+q}(D, \varepsilon _{p*}\mathbb Q_{D^p}[-p])\\ 
&=H^q(D^p, \mathbb Q)\\ 
&\Rightarrow H^{p+q}(D, \mathbb Q)
\end{align*}
such that 
the differential $d^{p, q}_1:{}_W\!E^{p, q}_1\to {}_W\!E^{p+1, q}_1$ is given by 
$$
d^{p, q}_1=\sum _{j=0}^{p+1}(-1)^j\lambda_{j, p+1}^*: H^q(D^p, \mathbb Q)
\to H^q(D^{p+1}, \mathbb Q)
$$
and it degenerates at $E_2$. 
The spectral sequence with respect to 
$F$ is 
\begin{align*}
{}_F\!E^{p, q}_1=\mathbb H^{p+q}(D, \Gr^p_F(s(\Omega^{\bullet}_{D^\bullet})))
\Rightarrow H^{p+q}(D, \mathbb C) 
\end{align*} 
and it degenerates at $E_1$. 
\end{say}

For the precise definitions of CHC and CMHC (CHMC, in French), 
see \cite[Section 8]{deligne} or 
\cite[Chapitre 3]{elzein2}. 
See also \cite[2.3.3 and 3.3]{ps}. 

The third example is not so standard but is indispensable for our injectivity 
theorems. 

\begin{say}\label{f-say4.3} 
Let $X$ be a complex manifold in Fujiki's class $\mathcal C$ 
and let $D$ be a simple normal crossing divisor on $X$. 
We consider the mixed cones of $\phi: \mathbb Q_X\to 
\mathbb Q_{D^{\bullet}}$ and $\psi:\Omega^\bullet _X\to 
\Omega^\bullet_{D^\bullet}$ with suitable shifts of complexes 
and weight filtrations (for the details, 
see, for example, \cite[I.3]{elzein}, 
\cite[3.7.14]{elzein2}, \cite[Section 3.3.4]{elzein-le} or 
\cite[Theorem 3.22]{ps}), 
where $\phi$ and $\psi$ are induced by the 
natural restriction map. 
More precisely, 
we define a complex $$\mathbb Q_{X-D^\bullet}=\Cone ^\bullet(\phi)[-1].$$ 
Then we have 
$$
(\mathbb Q_{X-D^\bullet})^p=(\mathbb Q_X)^p\oplus (\mathbb Q_{D^\bullet})^{p-1}. 
$$
The weight filtration on $\mathbb Q_{X-D^\bullet}$ is defined as follows:  
$$
(W_m\mathbb Q_{X-D^\bullet})^p
=(W_{m}\mathbb Q_X)^p\oplus (W_{m+1}(\mathbb Q_{D^\bullet}))^{p-1}. $$ 
We note that $\mathbb Q_{X-D^\bullet}$ is quasi-isomorphic to 
$j_{!}\mathbb Q_{X-D}$, 
where $j:X-D\to X$ is the natural open immersion.  
We put 
$$\Omega^\bullet_{X-D^\bullet}=\Cone ^\bullet(\psi)[-1]. 
$$ 
We note that 
$$
\Omega^p_{X-D^\bullet}=\Omega^p_X\oplus (s\Omega^\bullet_{D^\bullet})^{p-1}. 
$$
We define filtrations on $\Omega^\bullet _{X-D^\bullet}$ 
as follows: 
$$
(W_m\Omega^\bullet_{X-D^\bullet})^p=
(W_{m}\Omega^{\bullet}_X)^p\oplus (W_{m+1}(s\Omega^\bullet_{D^\bullet}))^{p-1}
$$ 
and 
$$(F^r\Omega^\bullet_{X-D^\bullet})^p=(F^r\Omega^{\bullet}_X)^p
\oplus (F^r(s\Omega^\bullet_{D^\bullet}))^{p-1}. 
$$ 
Then we obtain that the triple 
$$(j_!\mathbb Z_{X-D}, (\mathbb Q_{X-D^{\bullet}}, 
W), (\Omega^{\bullet}_{X-D^{\bullet}}, W, F))$$ is 
a CMHC. 
It defines a natural mixed Hodge structure on $H^{\bullet}_c(X-D, 
\mathbb Z)$. 
We note that 
$$
\Gr^W_0\mathbb Q_{X-D^\bullet}=\mathbb Q_X
$$ 
and 
$$\Gr^W_{-p}\mathbb Q_{X-D^\bullet}=\Gr^W_{1-p}\mathbb Q_{D^\bullet}[-1]
=\varepsilon _{p-1*}\mathbb Q _{D^{p-1}}[-p]
$$ for 
$p\geq 1$. 
The spectral sequence with respect to 
$W$ 
$$
{}_W\!E^{p, q}_1=H^{p+q}(X, \Gr ^W_{-p}\mathbb Q_{X-D^\bullet})
\Rightarrow H^{p+q}_c(X-D, \mathbb Q) 
$$ 
degenerates at $E_2$, 
where 
$$
{}_W\!E^{0, q}_1=H^q(X, \mathbb Q)
$$ and 
$$
{}_W\!E^{p, q}_1=H^q(D^{p-1}, \mathbb Q) 
$$ 
for every $p\geq 1$. We put 
$$
\Omega^\bullet _X(\log D)(-D)=\Omega^\bullet _X(\log D)\otimes \mathcal O_X(-D). 
$$
Since we can check that the complex 
\begin{eqnarray*}
0\to \Omega^{\bullet}_X(\log D)(-D)\to \Omega^{\bullet}_X
\to {\varepsilon_0}_*\Omega^{\bullet}_{D^0}\\
\to {\varepsilon_1}_*\Omega^{\bullet}_{D^1}\to 
\cdots 
\to {\varepsilon_k}_*\Omega^{\bullet}_{D^k}\to \cdots
\end{eqnarray*} 
is exact by direct local calculations, we see that 
$(\Omega^{\bullet}_{X-D^{\bullet}}, F)$ is quasi-isomorphic to 
$(\Omega^{\bullet}_X(\log D)(-D), F)$ in $D^+F(X, \mathbb C)$, 
where 
\begin{eqnarray*}
\lefteqn{F^p\Omega^{\bullet}_X(\log D)(-D)}\\ & & 
=(0\to \cdots \to 0 \to \Omega^p_X(\log D)(-D)\to 
\Omega^{p+1}_X(\log D)(-D)\to \cdots). 
\end{eqnarray*}
Therefore, the spectral sequence with respect to $F$ 
$$
E^{p,q}_1=H^q(X, \Omega^p_X(\log D)(-D))
\Rightarrow 
\mathbb H^{p+q}(X, \Omega^{\bullet}_X(\log D)(-D))
$$ 
degenerates at $E_1$. Note that the right hand side 
is isomorphic to $H^{p+q}_c(X-D, \mathbb C)$. 
We also note that 
$$
\Gr_F^0\Omega^\bullet_X(\log D)(-D)\simeq \mathcal O_X(-D). 
$$
\end{say}

\begin{rem}When we take {\em{mixed cones}} in \ref{f-say4.3}, 
we have to be careful about the commutativity 
of various comparison morphisms in the derived category 
(see \cite[Section 3.3.4]{elzein-le} and \cite[Remark 3.23]{ps}). 
\end{rem}

Let us recall the notion of {\em{$V$-manifolds}}. 
We need it for the proof of Theorem \ref{f-thm1.2}. 

\begin{defn}[$V$-manifold]\index{$V$-manifold}\label{f-def4.5}
A {\em{$V$-manifold}} of dimension $d$ is a complex analytic 
space that admits an open covering $\{U_i\}$ such 
that each $U_i$ is analytically isomorphic to $V_i/G_i$, 
where $V_i\subset \mathbb C^d$ is an 
open ball and $G_i$ is a finite subgroup of $\GL(d, \mathbb C)$. 
In this paper, $G_i$ is always an abelian group for every $i$. 

Let $X$ be a $V$-manifold and let $\Sigma$ be its 
singular locus. Then we define 
$$\widetilde {\Omega}^{\bullet}_X=
j_*\Omega^{\bullet}_{X-\Sigma}, $$ 
where 
$j:X-\Sigma \to X$ is the natural open immersion. 
A divisor $D$ on $X$ is called a {\em{divisor with 
$V$-normal crossings}} if locally on $X$ we have $(X, D)\simeq 
(V, E)/G$ with $V\subset \mathbb C^d$ an open domain, 
$G\subset \GL(d, \mathbb C)$ a small 
subgroup acting on $V$, and $E\subset V$ a $G$-invariant 
normal crossing divisor. 
We define $$\widetilde {\Omega}^{\bullet}_X(\log D)
=j_*\Omega^{\bullet}_{X-\Sigma}(\log D).$$ 
Furthermore, if $D$ is Cartier, then 
we put $$\widetilde {\Omega}^{\bullet}_X(\log D)(-D)
=\widetilde \Omega^{\bullet}_X(\log D)\otimes \mathcal O_X(-D). $$ 
\end{defn}

Let us go back to the proof of the $E_1$-degeneration of (\ref{f-shiki1}) in the proof of Theorem \ref{f-thm1.2}. 

\begin{proof}[Proof of the 
$E_1$-degeneration of $($\ref{f-shiki1}$)$ in the proof of Theorem \ref{f-thm1.2}]
\label{f-proof}
Here, we use the notation in the proof of 
Theorem \ref{f-thm1.2}. 
In this case, we know that $Y$ has only quotient singularities, that is, 
$Y$ is a $V$-manifold. 
We see that $Y$ is in Fujiki's class $\mathcal C$ (see Remark \ref{f-rem2.12}). 
Then we obtain that 
$$(\mathbb Z_Y, (\widetilde \Omega^{\bullet}_Y, F), \alpha)$$ 
is a CHC, where $F$ is the filtration b\^ete and 
$\alpha:\mathbb C_Y\to \widetilde {\Omega}^{\bullet}_Y$ is 
the inclusion. For the 
details, see \cite[(1.6)]{steenbrink}. 
By construction, $T$ is a divisor with $V$-normal crossings 
on $Y$ (see Definition \ref{f-def4.5} and \cite[(1.16) Definition]{steenbrink}). 
We can check that $Y$ is singular only over the singular 
locus of $\Supp B$. 
Let $\varepsilon:T^{\bullet}\to T$ be the Mayer--Vietoris 
simplicial resolution. Although each irreducible component of 
$T$ may have singularities, 
Definition \ref{f-def2.14} makes sense without any modifications. 
We note that 
$T^n$ has only quotient singularities for every $n\geq 0$ 
by the construction of $\pi:Y\to X$. 
We can also check that the same construction as in \ref{f-say4.2} 
works with only minor modifications. Hence we have  
a CMHC $$(\mathbb Z_T, (\mathbb Q_{T^{\bullet}}, W), 
(s(\widetilde {\Omega}^{\bullet}_{T^{\bullet}}), W, F))$$ that 
induces a natural mixed Hodge structure on $H^{\bullet}(T, \mathbb Z)$. 
By the same arguments as in \ref{f-say4.3}, we can construct a triple 
$$(j_!\mathbb Z_{Y-T}, (\mathbb Q_{Y-T^{\bullet}}, W), (K_{\mathbb C}, 
W, F)), 
$$ 
where $j:Y-T\to Y$ is the natural 
open immersion. 
It is a CMHC that induces a natural mixed Hodge structure on 
$H^{\bullet}_c
(Y-T, \mathbb Z)$ and $(K_{\mathbb C}, F)$ is quasi-isomorphic 
to $(\widetilde \Omega^{\bullet}_Y(\log T)(-T), F)$ 
in $D^+F(Y, \mathbb C)$, where 
\begin{eqnarray*}
\lefteqn{F^p\widetilde\Omega^{\bullet}_Y(\log T)(-T)}\\& &
=(0\to \cdots \to 0 \to \widetilde \Omega^p_Y(\log T)(-T)\to 
\widetilde\Omega^{p+1}_Y(\log T)(-T)\to \cdots). 
\end{eqnarray*} 
Therefore, the spectral sequence with respect to $F$ 
$$
E^{p,q}_1=H^q(Y, \widetilde\Omega^p_Y(\log T)(-T))
\Rightarrow 
\mathbb H^{p+q}(Y, \widetilde\Omega^{\bullet}_Y(\log T)(-T))
$$ 
degenerates at $E_1$. Note that 
the right hand side 
is isomorphic to $H^{p+q}_c(Y-T, \mathbb C)=H^{p+q}(Y, j_!\mathbb C_{Y-T})$. 
We also note that 
$$
\Gr^0_F\widetilde \Omega^{\bullet}_Y(\log T)(-T)\simeq \mathcal O_Y(-T). 
$$
\end{proof}

\section{Injectivity, torsion-free, and 
vanishing theorems}\label{f-sec5}

Once we establish Theorem \ref{f-thm1.2}, we can easily prove 
Theorem \ref{f-thm1.3}. 
Moreover, Theorem \ref{f-thm1.4} is an easy consequence of Theorem \ref{f-thm1.3}. 
 
\begin{proof}[Proof of Theorem \ref{f-thm1.3}] 
We can obtain Theorem \ref{f-thm1.3} as an application of 
Theorem \ref{f-thm1.2}. More precisely, 
by Theorem \ref{f-thm1.2}, the proof of \cite[Theorem 6.1]{fujino2} 
works with some suitable modifications. 
Note that the vanishing theorem of Reid--Fukuda type for birational 
morphisms (see \cite[Lemma 6.2]{fujino2}) holds for bimeromorphic 
morphisms between complex varieties by Lemma \ref{f-lem2.19}. 
The desingularization used in the proof of \cite[Theorem 6.1]{fujino2} 
holds also for compact complex varieties (see Remark \ref{f-rem2.5}). 
We leave the details as exercises for the reader. 
\end{proof}

\begin{proof}[Proof of Theorem \ref{f-thm1.4}]
Theorem \ref{f-thm1.4} follows from Theorem \ref{f-thm1.3} by some standard arguments. 
For (i), Step 1 in the proof of \cite[Theorem 6.3 (i)]{fujino2} is sufficient 
by Theorem \ref{f-thm1.3} since $X$ is projective. 
Step 1 in the proof of \cite[Theorem 6.3 (ii)]{fujino2} works 
without any modifications by Theorem \ref{f-thm1.3}. 
Therefore, we obtain the statement (ii). 
For the details, see \cite{fujino2}. 
\end{proof}

\section{Semipositivity theorem}\label{f-sec6}

The following result is a special case of Theorem \ref{f-thm1.4} (i). 

\begin{cor}\label{f-cor6.1}
Let $X$ be a compact K\"ahler manifold and 
let $f:X\to Y$ be a surjective morphism onto a projective variety $Y$. 
Let $D$ be a simple normal crossing divisor on $X$ such that 
every stratum of $D$ is dominant onto $Y$. 
Then $R^if_*\mathcal O_X(K_X+D)$ is torsion-free for every $i$. 
\end{cor}

Once we establish Corollary \ref{f-cor6.1}, 
we can prove Theorem \ref{f-thm1.5}. 

\begin{proof}[Proof of Theorem \ref{f-thm1.5}]
By Corollary \ref{f-cor6.1}, the arguments in \cite[Section 3]{fujino1} work 
without any modifications. Therefore, we obtain Theorem \ref{f-thm1.5}. 
\end{proof}

\begin{rem}
Theorem \ref{f-thm1.5} is contained in \cite{ffs}. 
See \cite[4.7.~Remark]{ffs}. 
Note that the argument in \cite{ffs} heavily depends on Saito's 
theory of mixed Hodge modules. 
\end{rem}

\begin{rem}
The semipositivity of $R^if_*\mathcal O_X(K_{X/Y}+D)$ 
in Theorem \ref{f-thm1.5} follows from \cite[Theorem 1.3]{fujino-fujisawa} 
or \cite[Theorem 3]{ffs}. 
We do not use \cite[Theorem 2]{kawamata2}, which depends on \cite{kawamata1}. 
For some comments on the semipositivity theorem in 
\cite{kawamata1}, see \cite[4.6.~Remark]{ffs}. 
\end{rem}

\begin{rem}\label{f-rem6.4}
In Theorem \ref{f-thm1.5}, the case when $i=0$ 
is similar to \cite[Theorem 32]{kawamata1}. 
Unfortunately, our results and arguments 
do not recover Kawamata's original statement (see \cite[Theorem 32]{kawamata1}). 
The author has been unable to 
follow \cite[Theorem 32]{kawamata1}. 
Anyway, our formulation of 
Theorem \ref{f-thm1.5} is natural 
and the statement of Theorem \ref{f-thm1.5} 
is sufficient for most applications. 
We do not use \cite[Theorem 32]{kawamata1} in this paper. 
Therefore, we do not touch \cite[Theorem 32]{kawamata1} here 
anymore. We note that 
\cite{kawamata1} was written before \cite{sz} and \cite{kashiwara}, 
where the theory of (admissible) 
variations of mixed Hodge structure was investigated. 
\end{rem}

\begin{rem}[On canonical extensions]
In the last paragraph in \cite[Section 2]{kawamata3}, 
Yujiro Kawamata wrote: 
\begin{quote}
The Hodge filtration $F$ of 
$\mathcal H$ extends to $\tilde {\mathcal H}$ such that $\mathrm{Gr}_F
(\tilde {\mathcal H})$ is still a locally free sheaf on $\tilde Y$. 
Indeed this is a consequence of the nilpotent 
orbit theorem [14] when $H$ is a variation of pure Hodge structures, and the 
general case follows immediately from this. 
\end{quote}
Note that [14] is Schmid's famous paper:~\cite{schmid}. 
Kawamata's statement is obviously wrong when $H$ is a 
variation of {\em{mixed}} Hodge structure. 
In \cite[5.~Canonical extension]{kawamata4}, he 
also wrote: 
\begin{quote}
By [17], the Hodge filtration of $\mathcal H$ extends to 
a filtration by locally free subsheaves, which we denote again by $F$. 
\end{quote} 
We note that [17] is Schmid's paper:~\cite{schmid}. 

As mentioned above, in \cite{kawamata3} and \cite{kawamata4}, Yujiro 
Kawamata seems to misuse Schmid's nilpotent orbit theorem. 
It is a result for polarizable variations of pure Hodge structure. 
We can not directly apply it to graded polarizable variations of mixed Hodge 
structure. For some explicit examples on this topic, 
see \cite[Example 1.5]{fujino-fujisawa}, \cite[(3.16) Example]{sz}, etc. 
Therefore, we do not use the results in \cite{kawamata3} and \cite{kawamata4} 
in this paper. 

Note that the main theorem of \cite{fujino-fujisawa} (see 
\cite[Theorem 1.1]{fujino-fujisawa}), whose proof is completely different from 
Kawamata's argument in \cite{kawamata3}, 
is sharper than Kawamata's main statement in \cite{kawamata3} (see 
\cite[Theorem 1.1]{kawamata3}). We also note that 
\cite[Theorem 1.1]{kawamata3} does not seem to cover the semipositivity 
theorem in \cite{fujino1} directly (see \cite[Theorem 3.9]{fujino1}). 
This is because \cite[Theorem 1.1]{kawamata3} requires some 
extra assumptions on every stratum. 
Roughly speaking, \cite[Theorem 3.9]{fujino1} 
is Theorem \ref{f-thm1.5} under the assumption that 
$X$ is a smooth projective variety. 
\end{rem}

Anyway, we have established Theorem \ref{f-thm1.5}, which is 
the main ingredient of the twisted weak positivity:~Theorem \ref{f-thm1.1}. 

\section{Weakly positive sheaves}\label{f-sec7} 

In this section, we discuss {\em{weakly positive sheaves}} introduced 
by Viehweg (see \cite{viehweg0} and \cite{viehweg1}). 
For the basic properties of weakly positive sheaves and related 
results, see, for example, \cite[Section 2]{viehweg-book} and 
\cite{fujino-revisited}. 
In this paper, 
we closely follow \cite{viehweg1}, \cite[Section 2]{viehweg-book}, 
\cite[Section 4.4]{campana}, and \cite{mori}. 
Here, we adopt \cite[Definition 3.1]{viehweg2} for the definition of 
weakly positive sheaves, which is slightly different from 
Viehweg's original definition (see \cite{viehweg0} and 
\cite[Definition 1.2]{viehweg1}). 

\begin{defn}
Let $W$ be a smooth projective 
variety and let $\mathcal F$ be a torsion-free 
coherent sheaf on $W$. 
Let $\widehat W$ be the largest Zariski open subset 
of $W$ such that 
$\mathcal F|_{\widehat W}$ is locally free. 
Then we put 
$$
\widehat S^k(\mathcal F)=i_*S^k(i^*\mathcal F) 
$$
where $i:\widehat W\to W$ is the natural open immersion and 
$S^k$ denotes the $k$-th symmetric product. 
Note that 
$\codim _W(W\setminus \widehat W)\geq 2$ since $\mathcal F$ is 
torsion-free. 
\end{defn}

\begin{defn}[Weak positivity]\label{f-def7.2}
Let $W$ be a smooth projective variety and let $\mathcal F$ be a torsion-free 
coherent sheaf on $W$. 
We call $\mathcal F$ {\em{weakly positive}}, 
if for every ample line bundle $\mathcal H$ on $W$ and every positive integer 
$\alpha$ there exists some positive integer $\beta$ such that 
$\widehat S^{\alpha\beta}(\mathcal F)\otimes \mathcal H^{\otimes \beta}$ is 
generically generated by global sections. 
This means that the natural map 
$$
H^0(W, \widehat S^{\alpha\beta}(\mathcal F)\otimes \mathcal H^{\otimes \beta})
\otimes \mathcal O_W\to \widehat S^{\alpha\beta}(\mathcal F)\otimes \mathcal H^{\otimes 
\beta}
$$ 
is generically surjective. 
By definition, the trivial sheaf $\mathcal F=0$ is weakly positive. 
\end{defn}

In some literature, the following definition is used for 
weakly positive sheaves on smooth projective varieties (see, 
for example, \cite{viehweg0}, and \cite[Definition 1.2]{viehweg1}). 

\begin{defn}[Original weak positivity]\label{f-def7.3}
Let $W$ be a smooth projective variety and let $\mathcal F$ be a torsion-free 
coherent sheaf on $W$. 
Let $U$ be a Zariski open set of $W$. 
We call $\mathcal F$ {\em{weakly positive over $U$}}, 
if for every ample line bundle $\mathcal H$ on $W$ and every positive integer 
$\alpha$ there exists some positive integer $\beta$ such that 
$\widehat S^{\alpha\beta}(\mathcal F)\otimes \mathcal H^{\otimes \beta}$ is 
generated by global sections over $U$. 
This means that the natural map 
$$
H^0(W, \widehat S^{\alpha\beta}(\mathcal F)\otimes \mathcal H^{\otimes \beta})
\otimes \mathcal O_W\to \widehat S^{\alpha\beta}(\mathcal F)\otimes \mathcal H^{\otimes 
\beta}
$$ 
is surjective over $U$. 
We call $\mathcal F$ {\em{weakly positive}}, if 
there exists some nonempty Zariski open set $U$ such that $\mathcal F$ is 
weakly positive over $U$. 
\end{defn}

Note that $U$ is independent of $\alpha$ and $\beta$ in Definition \ref{f-def7.3}. 
It is obvious that if $\mathcal F$ is weakly positive in the sense of 
Definition \ref{f-def7.3} then $\mathcal F$ is weakly positive in the 
sense of Definition \ref{f-def7.2}. 

In this paper, we do not use the weak positivity in the sense of Definition 
\ref{f-def7.3}, which is slightly stronger but is harder to prove 
than the weak positivity in the sense of Definition \ref{f-def7.2}. 

\begin{rem}\label{f-rem7.4} 
Let $W$ and $\mathcal F$ be as in Definition \ref{f-def7.2}. 
By Definition \ref{f-def7.2}, $\mathcal F$ is weakly positive if and 
only if so is $\widehat S^1(\mathcal F)=\mathcal F^{**}$, where 
$\mathcal F^{**}$ is the double-dual of $\mathcal F$. 
\end{rem} 

\begin{rem}[Nef locally free sheaf]\label{f-rem7.5}
Let $W$ be a smooth projective variety and let $\mathcal F$ be 
a locally free sheaf of finite rank on $W$. 
If 
$\mathcal F$ is nef, that is, 
$\mathcal F=0$ or 
$\mathcal O_{\mathbb P_W(\mathcal F)}(1)$ is a nef line 
bundle on $\mathbb P_W(\mathcal F)$,  then 
$\mathcal F$ is weakly positive. 

More precisely, let $\mathcal F\ne 0$ be a nef 
locally free sheaf of finite rank on a smooth projective variety $W$ and 
let $\mathcal H$ be an ample line bundle on $W$. 
We put $\pi:\mathbb P_W(\mathcal F)\to W$. 
Then we can easily check that 
$\mathcal O_{\mathbb P_W(\mathcal F)}(\alpha)\otimes \pi^*\mathcal H$ 
is an ample 
line bundle for every positive integer $\alpha$ by Nakai's criterion. 
In particular, $\mathcal O_{\mathbb P_W(\mathcal F)}(1)\otimes \pi^*\mathcal H$ 
is ample, equivalently, $\mathcal F\otimes \mathcal H$ is ample. 
By the argument b) in the proof of \cite[Proposition (3.2)]{hartshorne2}, 
there is a positive integer $\beta_0$ such that 
$S^{\alpha\beta}(\mathcal F)\otimes \mathcal H^{\otimes \beta}$ is 
generated by global sections for every integer $\beta\geq \beta_0$. 
Therefore, $\mathcal F$ is weakly positive. 
\end{rem}

\begin{rem}\label{f-rem7.6}
Let $\mathcal F$ be a line bundle on a smooth projective variety $W$. 
Then $\mathcal F$ is weakly positive 
if and only if $\mathcal F$ is pseudo-effective. 
\end{rem}

Although 
Lemma \ref{f-lem7.7} is obvious by the definition of weakly positive sheaves, 
it is very useful. We will repeatedly use Lemma \ref{f-lem7.7} in this 
section. 

\begin{lem}[{cf.~\cite[Lemma 1.4.~1)]{viehweg1}}]\label{f-lem7.7}
Let $W$ be a smooth projective variety and let $\mathcal F$ and $\mathcal G$ 
be torsion-free coherent sheaves on $W$. 
If $\mathcal F \to \mathcal G$ is a morphism 
which is generically surjective and if $\mathcal F$ is weakly positive, 
then $\mathcal G$ is also weakly positive. 
\end{lem}

Let us prove the following generalization of Viehweg's theorem 
(cf.~\cite[Theorem 4.1]{viehweg1}), 
which follows from Theorem \ref{f-thm1.5}. Viehweg only considered 
the case when $i=0$ with $D=0$. 

\begin{thm}[Fundamental weak positivity theorem]\label{f-thm7.8} 
Let $f:V\to W$ be a surjective morphism from a compact K\"ahler manifold 
$V$ to a smooth projective variety $W$. 
Let $D$ be a simple normal crossing divisor on $V$ such that 
every irreducible component of $D$ is dominant onto $W$. 
Let $\Sigma$ be a simple normal crossing divisor on $W$ such that 
$f$ is smooth over $W_0=W\setminus \Sigma$ and 
that $D$ is relatively normal crossing over $W_0$, 
and $\Supp (D+f^*\Sigma)$ is a simple normal crossing 
divisor on $V$. 
Then 
the locally free sheaf $R^if_*\mathcal O_V(K_{V/W}+D)$ is weakly positive 
for every $i\geq 0$.  
\end{thm}

\begin{proof}
We put $V_0=f^{-1}(W_0)$, $f_0=f|_{V_0}$, $D_0=D|_{V_0}$, and 
$d=\dim V-\dim W$. We take a finite morphism $g:W'\to W$ from a smooth 
projective variety, which induces a unipotent reduction for the 
local system $R^{d+i}f_{0*}\mathbb C_{V_0- D_0}$, such 
that $\Supp (g^*\Sigma)$ is a simple normal crossing divisor on $W'$ 
(see, for example, \cite[Proposition 2.67]{kollar-mori} 
and \cite[Theorem 17]{kawamata1}). 
We can construct a commutative diagram: 
$$
\xymatrix{
V' \ar[r] \ar[d]_{f'} & V\ar[d]^{f} \\
   W' \ar[r]_{g} & W
} 
$$
with the following properties: 
\begin{itemize}
\item[(i)] $V'$ is a compact K\"ahler manifold which 
is a resolution of $V\times _W W'$, 
\item[(ii)] $f'$ is smooth over $W'_0=W'\setminus \Sigma'$, where 
$\Sigma'=\Supp (g^*\Sigma)$, 
\item[(iii)] $D'$ is a simple normal crossing divisor on $V'$ such that 
$D'$ and $f'^*\Sigma'$ have no common irreducible components and 
that $\Supp (D'+f'^*
\Sigma')$ is a simple normal crossing divisor on $V'$, and 
\item[(iv)] $f': (V', D')\to W'$ is nothing but the base change of $f:(V, D)\to W$ over 
$W'_0$. 
\end{itemize}
Then we obtain a natural inclusion 
of locally free sheaves 
$$
\varphi^i:R^if'_*\mathcal O_{V'}
(K_{V'/W'}+D')\subset g^*R^if_*\mathcal O_V(K_{V/W}
+D). 
$$
such that $\varphi^i$ is the identity over $W'_0$. 
Note that $R^if_*\mathcal O_V(K_{V/W}+D)$ is the upper canonial 
extension of the bottom Hodge filtration 
and that $R^if'_*\mathcal O_{V'}(K_{V'/W'}+D')$ is 
the canonical extension of the bottom Hodge filtration 
(see Theorem \ref{f-thm1.5}). 
By Theorem \ref{f-thm1.5}, $R^if'_*\mathcal O_{V'}(K_{V'/W'}+D')$ is nef. 
In particular, $R^if'_*\mathcal O_{V'}(K_{V'/W'}+D')$ is weakly positive 
(see Remark \ref{f-rem7.5}). Let $H$ be an ample Cartier divisor on 
$W$ and let $\alpha$ be a positive integer. 
Then 
$$
S^{2\alpha\beta}(g^*R^if_*\mathcal O_V(K_{V/W}+D))\otimes 
\mathcal O_{W'}(\beta g^*H) 
$$ 
is generically generated by global sections for some positive integer 
$\beta$ since 
$g^*R^if_*\mathcal O_V(K_{V/W}+D)$ is 
weakly positive by Lemma \ref{f-lem7.7}. 
Without loss of generality, we may assume that 
$\mathcal O_W(\beta H)\otimes g_*\mathcal O_{W'}$ is 
generated by global sections. 
We have a surjection 
$$
g_*g^*S^{2\alpha \beta}(R^if_*\mathcal O_V(K_{V/W}+D))
\to S^{2\alpha \beta}(R^if_*\mathcal O_V(K_{V/W}+D)). 
$$ 
Therefore, we have a generically surjective homomorphism 
$$
\bigoplus _{\mathrm{finite}} \mathcal O_{W}(\beta H)\otimes 
g_*\mathcal O_{W'}\to 
S^{2\alpha \beta}(R^if_*\mathcal O_V(K_{V/W}+D))\otimes \mathcal O_W(2\beta H). 
$$ 
Thus, we obtain that  
$$
S^{2\alpha \beta}(R^if_*\mathcal O_V(K_{V/W}+D))\otimes \mathcal O_W(2\beta H)
$$ 
is generically generated by global sections. This means that 
the locally free sheaf 
$R^if_*\mathcal O_V(K_{V/W}+D)$ is weakly positive. 
\end{proof}

\begin{rem}
In general, the locally free sheaf 
$R^if_*\mathcal O_V(K_{V/W}+D)$ is not necessarily 
nef. See \cite[Section 8]{fujino-fujisawa} for some examples. 
\end{rem}

\begin{rem}
In Theorem \ref{f-thm7.8}, 
it is easy to see that $R^if_*\mathcal O_V(K_{V/W}+D)$ is weakly positive 
over $W_0$ in the sense of Definition \ref{f-def7.3}.  
\end{rem}

By using the basic properties of weakly positive sheaves, we can 
obtain the following corollary of 
Theorem \ref{f-thm7.8}, which is new when $i>0$. 

\begin{cor}\label{f-cor7.11} 
Let $f:V\to W$ be a surjective morphism from a compact K\"ahler manifold 
$V$ to a smooth projective variety $W$. 
Let $D$ be a simple 
normal crossing divisor on $V$. 
Then 
the torsion-free part of 
$R^if_*\mathcal O_V(K_{V/W}+D)$, that is, 
$$
R^if_*\mathcal O_V(K_{V/W}+D)/{\mathrm{torsion}}, 
$$ 
is weakly positive. 
\end{cor}
\begin{proof}
By replacing $D$ with its horizontal part, we may assume that 
every irreducible component of $D$ is dominant 
onto $W$ (see Lemma \ref{f-lem7.7}). 
If there is a log canonical center $C$ of $(V, D)$ such that 
$f(C)\subsetneq W$, then we take the blow-up 
$h:V'\to V$ along $C$. 
We put 
$$
K_{V'}+D'=h^*(K_V+D). 
$$ 
Then $D'$ is a simple normal crossing divisor on $V'$ and 
$$
R^if_*\mathcal O_V(K_{V/W}+D)\simeq R^i(f\circ h)_*\mathcal O_{V'}(K_{V'/W}+D')
$$ 
for every $i$. 
Therefore, we can replace $(V, D)$ with $(V', D')$. 
Then we replace $D$ with its horizontal part (see 
Lemma \ref{f-lem7.7}). 
By repeating this process finitely many times, 
we may assume that every stratum of $D$ is dominant onto $W$. 
In this case, $R^if_*\mathcal O_V(K_{V/W}+D)$ is torsion-free 
by Theorem \ref{f-thm1.4} (i). 
Now we take a closed subset $\Sigma$ of $W$ such that 
$f$ is smooth over $W\setminus \Sigma$ and that $D$ is relatively 
normal crossing over $W\setminus \Sigma$. 
Let $g:W'\to W$ be a birational morphism from 
a smooth projective variety $W'$ such that $\Sigma'=g^{-1}(\Sigma)$ is a simple 
normal crossing divisor. By taking some suitable 
blow-ups of $V$ in $f^{-1}(\Sigma)$ and 
replacing $D$ with its strict transform, we may further assume the following 
conditions: 
\begin{itemize}
\item[(i)] $f'=g^{-1}\circ f:V\to W'$ is a morphism, 
\item[(ii)] $f'$ is smooth over $W'\setminus \Sigma'$ and $D$ is relatively 
normal crossing over $W'\setminus \Sigma'$, and 
\item[(iii)] every irreducible component of $D$ is dominant 
onto $W$ and $\Supp (f'^*\Sigma'+D)$ is a simple normal crossing divisor 
on $V$. 
\end{itemize}
$$
\xymatrix{
V\ar[d]_{f'} \ar[dr]^{f}& \\
   W' \ar[r]_{g} & W
} 
$$
Here we used Szab\'o's resolution lemma (see Remark \ref{f-rem2.5}) and 
Lemma \ref{f-lem2.19}. 
Then, by Theorem \ref{f-thm7.8}, 
$R^if'_*\mathcal O_V(K_{V/W'}+D)$ is weakly positive. 
Note that 
$$
R^if'_*\mathcal O_V(K_{V/W}+D)\simeq R^if'_*\mathcal 
O_V(K_{V/W'}+D)\otimes \mathcal O_{W'}(E)
$$ 
where $E$ is a $g$-exceptional effective divisor such that $K_{W'}
=g^*K_W+E$.
Thus $R^if'_*\mathcal O_V(K_{V/W}+D)$ is weakly positive. 
We note that 
$$g_*R^if'_*\mathcal O_{V}(K_{V/W}+D)\simeq R^if_*\mathcal O_V(K_{V/W}+D).$$  
Here we used the fact that 
$$
R^pg_*R^qf'_*\mathcal O_V(K_{V/W}+D)=0
$$ 
for every $p>0$ and $q\geq 0$ by Proposition \ref{f-prop7.15} below. 
We can take an effective $g$-exceptional divisor $F$ on $W'$ such that 
$-F$ is $g$-ample. 
Let $H$ be an ample Cartier divisor 
on $W$. 
Then there exists a positive integer $k$ such that 
$kg^*H-F$ is ample. 
Since $R^if'_*\mathcal O_V(K_{V/W}+D)$ is weakly positive, 
$$
S^{k\alpha\beta}(R^if'_*\mathcal O_V(K_{V/W}+D))\otimes 
\mathcal O_{W'}(\beta(kg^*H-F))
$$ 
is generically generated by 
global sections. 
By taking $g_*$, 
$$
\widehat {S}^{\alpha k \beta}(R^if_*\mathcal O_V(K_{V/W}+D))\otimes 
\mathcal O_{W}(k\beta H)
$$ 
is generically generated by global sections. 
This means that the torsion-free sheaf 
$R^if_*\mathcal O_V(K_{V/W}+D)$ is weakly positive. 
\end{proof}

\begin{rem} In the proof of Corollary \ref{f-cor7.11}, 
the following isomorphism 
$$
\left( g_*R^if'_*\mathcal O_{V}(K_{V/W}+D)\right)^{**}
\simeq \left(R^if_*\mathcal O_V(K_{V/W}+D)\right)^{**}
$$ 
is obvious since $g$ is birational. 
This isomorphism is sufficient for the proof of 
the weak positivity of $R^if_*\mathcal O_{V}(K_{V/W}+D)$ (see 
Remark \ref{f-rem7.4}) although 
we used a shaper isomorphism 
$$
g_*R^if'_*\mathcal O_{V}(K_{V/W}+D)
\simeq R^if_*\mathcal O_V(K_{V/W}+D). 
$$
\end{rem}

\begin{rem}
In Corollary \ref{f-cor7.11}, 
we take a Zariski open set $W_0$ of $W$ such that 
$f$ is smooth over $W_0$ and that $D$ is relatively normal crossing 
over $W_0$. 
Then we see that 
$$R^if_*\mathcal O_V(K_{V/W}+D)/\mathrm{torsion}$$ is weakly 
positive over $W_0$ in the sense of Definition \ref{f-def7.3} by the proof of 
Corollary \ref{f-cor7.11}.  
\end{rem}

Before we give a proof of Proposition \ref{f-prop7.15}, 
we prepare an easy lemma. 

\begin{lem}\label{f-lem7.14} 
Theorem \ref{f-thm1.4} (ii) holds true even when 
$Y$ is a complex manifold in Fujiki's class $\mathcal C$ and $\Delta$ is an effective $\mathbb R$-divisor on $Y$ such that $(Y, \Delta)$ is dlt. 
\end{lem}

\begin{proof}
Let $h:Y'\to Y$ be a resolution such that $h$ is an isomorphism 
over the general points of any log canonical center of $(Y, \Delta)$ (see 
Remark \ref{f-rem2.5}). 
We can write 
$$
K_{Y'}+\Delta_{Y'}=h^*(K_Y+\Delta)+E
$$
where $\Delta_{Y'}$ and $E$ are effective, 
$E$ is Cartier and $h$-exceptional, 
$\Delta_{Y'}$ is a boundary $\mathbb R$-divisor, and $\Supp (\Delta_{Y'}+E)$ is 
a simple normal crossing divisor on $Y'$. 
Then \begin{align*}
h^*\mathcal L\otimes \mathcal O_{Y'}(E)-(K_{Y'}+\Delta_{Y'})&=
h^*(\mathcal L-(K_Y+\Delta))\\
&\sim _{\mathbb R} h^*f^*H. 
\end{align*}
By Theorem \ref{f-thm1.4} (ii), we obtain that 
$$
H^p(X, R^q(f\circ h)_*(h^*\mathcal L\otimes \mathcal O_{Y'}(E)))=0
$$ 
for every $p>0$ and $q\geq 0$. 
Note that $R^ih_*(h^*\mathcal L\otimes \mathcal O_{Y'}(E))=0$ for every $i>0$ 
by Lemma \ref{f-lem2.19}. 
Thus we obtain 
\begin{align*}
H^p(X, R^qf_*\mathcal L)\simeq 
H^p(X, R^q(f\circ h)_*(h^*\mathcal L\otimes \mathcal O_{Y'}(E)))=0
\end{align*}
for every $p>0$ and $q\geq 0$. Note that 
$h_*(h^*\mathcal L\otimes \mathcal O_{Y'}(E))\simeq \mathcal L$. 
\end{proof}

\begin{prop}\label{f-prop7.15} 
Let $f:X\to Y$ be a surjective morphism such that 
$X$ is a complex manifold in Fujiki's class $\mathcal C$ and 
$Y$ is a projective variety. 
Let $D$ be a simple normal crossing divisor on $X$ such that 
every stratum of $D$ is dominant onto $Y$. 
Let $g:Y\to Z$ be a birational morphism 
between projective varieties. 
Then 
$$
R^pg_*R^qf_*\mathcal O_X(K_X+D)=0
$$ 
for every $p>0$ and $q\geq 0$. 
\end{prop}

\begin{proof}
Let $\mathcal A$ be a sufficiently ample line bundle 
on $Z$ such that 
$$
H^r(Z, R^pg_*R^qf_*\mathcal O_X(K_X+D)\otimes \mathcal A)=0
$$ 
for $p>0, q\geq 0$, and $r>0$ 
and that 
$$
R^pg_*R^qf_*\mathcal O_X(K_X+D)\otimes \mathcal A
$$ 
is generated by global sections for $p>0$ and $q\geq 0$. 
By the Leray spectral sequence, we have 
\begin{align*}
&H^0(Z, R^pg_*R^qf_*\mathcal O_X(K_X+D)\otimes \mathcal A)
\\ &\simeq H^p(Y, R^qf_*\mathcal O_X(K_X+D+f^*g^*\mathcal A)).  
\end{align*} 
Therefore, it is sufficient to prove that 
$$
H^p(Y, R^qf_*\mathcal O_X(K_X+D+f^*g^*\mathcal A))=0
$$ 
for $p>0$ and $q\geq 0$. 
By Kodaira's lemma, we can write 
$$
g^*\mathcal  A\sim _{\mathbb Q} H+E
$$  
where $H$ is an ample $\mathbb Q$-divisor 
on $Y$ and $E$ is an effective $\mathbb Q$-Cartier $\mathbb Q$-divisor on $Y$. 
Let $\varepsilon$ be a sufficiently small positive number. 
Then $(X, D+\varepsilon f^*E)$ is dlt and 
\begin{align*}
\mathcal O_X(K_X+D+f^*g^*\mathcal A)-(K_X+D+\varepsilon f^*E)\sim_
{\mathbb Q}
f^*((1-\varepsilon)g^*\mathcal A+\varepsilon H). 
\end{align*} 
Therefore, by Lemma \ref{f-lem7.14}, we obtain that 
$$
H^p(Y, R^qf_*\mathcal O_X(K_X+D+f^*g^*\mathcal A))=0 
$$ 
for $p>0$ and $q\geq 0$. 
\end{proof}

\section{Twisted weak positivity}\label{f-sec8}

This section is the main part of this paper. Here, 
we prove the twisted weak positivity 
theorem:~Theorem \ref{f-thm1.1}. 

Lemma \ref{f-lem8.1} is a slight generalization of \cite[Lemma 5.1]{viehweg1}. 
It follows from Corollary \ref{f-cor7.11} by the usual covering trick. 

\begin{lem}[{cf.~\cite[Lemma 5.1]{viehweg1}}]\label{f-lem8.1}
Let $f:X\to Y$ be a surjective morphism from a compact 
K\"ahler manifold $X$ to a smooth projective variety $Y$. 
Let $D$ be a simple normal crossing divisor on $X$. 
Let $\mathcal L$ and $\mathcal N$ be line bundles on 
$X$ and let $C$ be an effective divisor on $X$ such that 
$\mathcal L^{N}=\mathcal N+C$ for some positive integer $N$, 
$D$ and $C$ have no common components, and $\Supp (D+C)$ is 
a simple normal crossing divisor on $X$. 
Assume that there is a nonempty Zariski open set $U$ of $Y$ such that 
some power of $\mathcal N$ is generated over $f^{-1}(U)$ by global sections. 
Then the sheaf 
$$
f_*\mathcal O_X(K_{X/Y}+D+\mathcal L^{(i)}) 
$$ 
is weakly positive for $0\leq i\leq N-1$, where 
$$
\mathcal L^{(i)}=\mathcal L^i\otimes \mathcal O_X\left(-\left\lfloor \frac {iC}{N}\right\rfloor\right). 
$$
\end{lem}

\begin{proof}
Since the statement is compatible with replacing $N$ by $NN'$, 
$C$ by $N'C$, and $\mathcal N$ by $\mathcal N^{N'}$ for some 
positive integer $N'$, 
we may assume that $\mathcal N$ itself is generated by 
global sections over $f^{-1}(U)$. 
This means that the base locus of $|\mathcal N|$ is contained in 
$X\setminus f^{-1}(U)$. 
Without loss of generality, we may shrink $U$ if necessary. 
Let $B+F$ be the zero set of a general section of $\mathcal N$ such that 
every irreducible component of $B$ is dominant onto $Y$ and 
that $\Supp F\subset X\setminus f^{-1}(U)$. 
By Bertini's theorem, 
$B$ is smooth and $\Supp (B+D+C)$ is a simple normal crossing 
divisor on $f^{-1}(U)$. 
We note that $\mathcal N=\mathcal O_X(B+F)$. 
By taking a suitable bimeromorphic modification outside $f^{-1}(U)$, 
we may assume that $B$ is smooth and that $\Supp (B+D+C+F)$ is a simple 
normal crossing 
divisor (see Remark \ref{f-rem2.5}). In fact, if $h:\widetilde X\to X$ is a bimeromorphic 
modification which is an isomorphism over 
$f^{-1}(U)$ and if $\widetilde {\mathcal L}
=h^*\mathcal L$, 
$\widetilde {\mathcal N}
=h^*\mathcal N$, $\widetilde C=h^*C$, and $\widetilde D$ is the strict 
transform of $D$, 
then we can easily check that 
$h_*\mathcal O_{\widetilde X}(K_{\widetilde X/Y}+
\widetilde D+\widetilde {\mathcal L}^{(i)})$ 
is contained in $\mathcal O_X(K_{X/Y}+D+\mathcal L^{(i)})$. 
By construction, 
$h_*\mathcal O_{\widetilde X}(K_{\widetilde X/Y}+
\widetilde D+\widetilde {\mathcal L}^{(i)})$ 
coincides with $\mathcal O_X(K_{X/Y}+D+\mathcal L^{(i)})$ on $f^{-1}(U)$. 
When we prove 
the weak positivity of $f_*\mathcal O_X(K_{X/Y}+D+\mathcal 
L^{(i)})$, by replacing $C$ with $C+F$, we may assume that $\mathcal L^N
=\mathcal O_X(B+C)$, that is, $F=0$ 
(see Lemma \ref{f-lem7.7}). Note that 
every irreducible component of $F$ is vertical with respect to 
$f:X\to Y$. 
By taking a cyclic cover $p:Z'\to X$ 
associated to $\mathcal L^{N}=\mathcal O_X(B+C)$, that is, 
$Z'$ is the normalization of $\Spec \bigoplus _{i=1}^{N-1}\mathcal L^{-i}$. 
Let $Z$ be a 
resolution of the cyclic cover $Z'$ 
and let $g:Z\to Y$ be the corresponding morphism. 
$$
\xymatrix{
Z\ar[r]^{q} \ar[rrd]_{g} & Z'\ar[r]^{p} &X\ar[d]^{f} \\
  && Y
} 
$$
It is well-known that $Z'$ has only quotient singularities 
and 
$$
p_*\mathcal O_{Z'}(K_{Z'})
\simeq \bigoplus _{i=0}^{N-1}\mathcal O_X(K_X+\mathcal L^{(i)}). 
$$
Let $D^{\dag}$ be the union of the strict transform of $p^*D$ and 
the exceptional divisor of $q:Z\to Z'$. 
Then $$
q_*\mathcal O_Z(K_{Z}+D^\dag)\simeq \mathcal O_{Z'}(K_{Z'}+p^*D). 
$$
Note that $(Z', p^*D)$ is log canonical. 
Of course, we may assume that $D^\dag$ is a simple normal crossing 
divisor. We obtain 
$$
g_*\mathcal O_Z(K_{Z/Y}+ D^\dag)\simeq \bigoplus_{i=0}^{N-1}f_*\mathcal 
O_X(K_{X/Y}+D+\mathcal L^{(i)}).  
$$ 
Therefore, by Corollary \ref{f-cor7.11}, $f_*\mathcal O_X(K_{X/Y}+D+\mathcal 
L^{(i)})$ is weakly positive for every $i$. 
\end{proof}

Before we start the proof of Theorem \ref{f-thm1.1}, we prepare 
a very important lemma. 

\begin{lem}[{cf.~\cite[Lemma 4.19]{campana}}]\label{f-lem8.2}
Let $f:X\to Y$ be a surjective morphism from a compact K\"ahler manifold 
$X$ to a smooth projective variety $Y$. 
Let $\Delta$ be a boundary $\mathbb Q$-divisor 
on $X$ such that $\Supp \Delta$ is a simple normal 
crossing divisor. 
Let $l$ be a positive integer such that $l(K_{X/Y}+\Delta)$ is Cartier. 
Let $A'$ be an ample Cartier divisor on $Y$. 
We put $A=f^*A'$. 
Assume that 
$$
\widehat S^N(f_*\mathcal O_X(l(K_{X/Y}+\Delta)+lA))
$$ 
is generated by global sections on some nonempty Zariski open set of $Y$. 
Then $$f_*\mathcal O_X(l(K_{X/Y}+\Delta)+(l-1)A)$$ is weakly positive. 
\end{lem}

\begin{proof}
We consider 
$$
\mathcal M:=\mathrm{Im} (f^*f_*\mathcal O_X(l(K_{X/Y}+\Delta)+lA)
\to \mathcal O_X(l(K_{X/Y}+\Delta)+lA)). 
$$
We may assume that 
the relative base locus of $l(K_{X/Y}+\Delta)+lA$ does not contain any 
component of $l\Delta$ by decreasing the relevant coefficients of $\Delta$. 
Furthermore, if necessary,  by taking blow-ups of $X$ and decreasing 
the relevant coefficients of $\Delta$, 
we may assume that $\mathcal M$ is a line bundle, 
$l(K_{X/Y}+\Delta)+lA=\mathcal M+E$, where 
$E$ is an effective divisor on $X$, $E$ and $l\Delta$ 
have no common components, and $\Supp (E+\Delta)$ is a simple 
normal crossing divisor on $X$. 
Let $$f:X\overset{\psi}{\longrightarrow} V\longrightarrow Y$$
be the Stein factorization. 
Then 
$$
\mathcal M=\mathrm{Im} (\psi^*\psi_*\mathcal O_X(l(K_{X/Y}+\Delta)+lA)
\to \mathcal O_X(l(K_{X/Y}+\Delta)+lA)). 
$$ 
If we take a Zariski open set $U'$ of $Y$ such that 
$\mathrm{codim}_Y(Y\setminus U')\geq 2$ suitably, then 
$\psi_*\mathcal O_X(l(K_{X/Y}+\Delta)+lA)$ is locally 
free and $\psi_*\mathcal O_X(l(K_{X/Y}+\Delta)+lA)\simeq 
\psi_*\mathcal M$ over $U'$ by construction. 
Therefore, $f_*\mathcal O_X(l(K_{X/Y}+\Delta)+lA)\simeq 
f_*\mathcal M$ on $U'$. 
Since 
$$
\widehat S^N(f_*\mathcal O_X(l(K_{X/Y}+\Delta)+lA))
$$ 
is generically generated by global sections, 
there is an effective divisor $S$ on $X$ such that 
$\mathrm{codim}_Yf(S)\geq 2$ and that 
$\mathcal M^{N}\otimes \mathcal O_X(NlS)$ is generated by 
global sections over $f^{-1}(U)\subset X$, where $U$ is a nonempty Zariski 
open set of $Y$. 
We put 
$$
L=K_{X/Y}+\Delta^{=1}+l\{\Delta\}+A+S 
$$ 
and 
$$
L^{(l-1)}=(l-1)L-\left\lfloor \frac{l-1}{l}(E+(l-1)l\{\Delta\})\right\rfloor. 
$$ 
We note that 
\begin{align*}
\Delta^{=1}=\lfloor \Delta\rfloor {\quad {\text{and}}\quad } \Delta=\Delta^{=1}+\{\Delta\}
\end{align*} 
because $\Delta$ is a boundary $\mathbb Q$-divisor. 
We also note that 
\begin{align*}
lL&=lK_{X/Y}+l\Delta+lA+lS+(l-1)l\{\Delta\}\\ 
&= \mathcal M+lS+E+(l-1)l\{\Delta\}. 
\end{align*}
By the usual covering 
argument, we obtain that 
$$
f_*\mathcal O_X(K_{X/Y}+\Delta^{=1}+L^{(l-1)})
$$ 
is weakly positive by Lemma \ref{f-lem8.1}. 
We can easily see that 
\begin{align*}
&K_{X/Y}+\Delta^{=1}+(l-1)L-\lfloor (l-1)^2 \{\Delta\}\rfloor
\\&=l(K_{X/Y}+\Delta^{=1})+l(l-1)\{\Delta\}-\lfloor (l^2-2l+1)\{\Delta\}\rfloor +(l-1)(A+S)
\\&=l(K_{X/Y}+\Delta)+(l-1)(A+S). 
\end{align*}
Therefore, we can also check that 
\begin{align*}
&f_*\mathcal O_X(K_{X/Y}+\Delta^{=1}+L^{(l-1)})\\ 
&\subset 
f_*\mathcal O_X(K_{X/Y}+\Delta^{=1}+(l-1)L-\lfloor (l-1)^2\{\Delta\}\rfloor)
\end{align*} 
and that they coincide over the generic point of $Y$ 
because $E$ is the relative base locus of $l(K_{X/Y}+\Delta)+lA$, 
$A=f^*A'$, $f(S)\subsetneq Y$, and 
$$
L^{(l-1)}=(l-1)L-\lfloor (l-1)^2\{\Delta\}\rfloor -\left\lfloor \frac{l-1}{l}E\right\rfloor . 
$$ 
Hence we obtain that 
\begin{align*}
&f_*\mathcal O_X(l(K_{X/Y}+\Delta)+(l-1)(A+S))\\ &=
f_*\mathcal O_X(K_{X/Y}+\Delta^{=1}+(l-1)L-\lfloor (l-1)^2\{\Delta\}\rfloor)
\end{align*} 
is 
weakly positive by Lemma \ref{f-lem7.7}. 
Therefore, 
$$(f_*\mathcal O_X(l(K_{X/Y}+\Delta)+(l-1)A))^{**}$$ 
is weakly positive because $\mathrm{codim}_Yf(S)\geq 2$. 
This means that 
$$f_*\mathcal O_X(l(K_{X/Y}+\Delta)+(l-1)A)$$
is weakly positive (see Remark \ref{f-rem7.4}).  
\end{proof}

Let us start the proof of Theorem \ref{f-thm1.1}. 

\begin{proof}[Proof of Theorem \ref{f-thm1.1}] 
We divide the proof into several steps. 
\begin{step}\label{step1}
Let $\widetilde X\to X$ be a resolution such that $\widetilde X$ is a compact K\"ahler 
manifold with $$
K_{\widetilde X}+\widetilde \Delta=\pi^*(K_X+\Delta)+E, 
$$ 
where $E$ and $\widetilde \Delta$ are effective and 
have no common components. 
By replacing $(X, \Delta)$ with $(\widetilde X, \widetilde \Delta)$, 
we may assume that $X$ is a compact K\"ahler manifold and that $\Supp \Delta$ 
is a simple normal crossing divisor. 
By replacing $mk$ with $k$, we may assume that 
$m=1$. 
\end{step}
\begin{step}\label{step2}
Let $H$ be an ample Cartier divisor on $Y$. 
We put 
$$
r=\min\{s>0 \, ; f_*\mathcal O_X(k(K_{X/Y}+\Delta))\otimes \mathcal O_Y((sk-1)H) 
\ \text{is weakly positive}\}. 
$$ 
By definition, we can find $\nu>0$ such that 
$$
\widehat S^{\nu}(f_*\mathcal O_X(k(K_{X/Y}+\Delta)))\otimes \mathcal O_Y((rk\nu-\nu)H)
\otimes \mathcal O_Y(\nu H)
$$ 
is generated by global sections over a nonempty Zariski open set. 
By Lemma \ref{f-lem8.2}, 
$$
f_*\mathcal O_X(k(K_{X/Y}+\Delta)\otimes \mathcal O_Y((rk-r)H)
$$ 
is weakly positive. 
The choice of $r$ allows this only if $(r-1)k-1<rk-r$, equivalently, 
$r\leq k$. 
Hence we obtained the weak positivity of $$f_*\mathcal O_Y(k(K_{X/Y}+\Delta))\otimes 
\mathcal O_Y((k^2-k)H). $$ 
\end{step}
\begin{step}\label{step3} 
Let $d$ be an arbitrary positive integer. 
By Lemma \ref{f-lem8.3}, 
we can take a finite flat morphism $g:Y'\to Y$ from a smooth projective 
variety $Y'$ such that 
$g^*H\sim dH'$ and that 
$X'=X\times _Y Y'$ is a compact K\"ahler manifold. 
We put $\tau:X'\to X$ and $\Delta'=\tau^*\Delta$. 
Thus, by Lemma \ref{f-lem8.3}, we have    
$$ 
f'_*\mathcal O_{X'}(k(K_{X'/Y'}+\Delta'))\simeq  g^*f_*\mathcal O_X(k(K_{X/Y}+\Delta)), 
$$ 
where 
$f':X'\to Y'$. 
We may further assume that $\Delta'$ is a boundary 
$\mathbb Q$-divisor such that 
$\Supp \Delta'$ is a simple normal crossing divisor 
by Lemma \ref{f-lem8.3}. 
Then we obtain that 
$$
g^*f_*\mathcal O_X(k(K_{X/Y}+\Delta))\otimes \mathcal O_{Y'}((k^2-k)H') 
$$ 
is weakly positive. 
This is because 
$$
f'_*\mathcal O_{X'}(k(K_{X'/Y'}+\Delta'))\otimes \mathcal O_{Y'}((k^2-k)H')
$$
is weakly 
positive by applying the above result (see Step \ref{step2}) to $f':(X', \Delta')\to Y'$. 
If $\alpha$ is a positive integer, then we choose 
$d=2\alpha (k^2-k)+1$. 
Let $\beta$ be a sufficiently large positive integer. 
Then we have that 
\begin{align*}
&\widehat S^{2\alpha\beta}(g^*f_*\mathcal O_X(k(K_{X/Y}+\Delta))\otimes \mathcal O_{Y'}((k^2-k)H'))\otimes 
\mathcal O_{Y'}(\beta H') \\&= 
g^*\widehat S^{2\alpha\beta}(f_*\mathcal O_X(k(K_{X/Y}+\Delta)))\otimes 
g^*\mathcal O_Y(\beta H)
\end{align*} 
is generated by global sections over a nonempty Zariski open set. 
We may further assume that 
$g_*\mathcal O_{Y'}\otimes \mathcal O_Y(\beta H)$ is 
generated by global sections. 
Over the Zariski open set $\widehat Y$ of $Y$ where 
$$
\widehat S^{2\alpha\beta}(f_*\mathcal O_X(k(K_{X/Y}+\Delta)))
$$ 
is locally free, we have a surjection
\begin{align*}
g_*g^*\widehat S^{2\alpha\beta}(f_*\mathcal O_X(k(K_{X/Y}+\Delta)))
\otimes \mathcal O_Y(\beta H) \\ \to 
\widehat S^{2\alpha\beta}(f_*\mathcal O_X(k(K_{X/Y}+\Delta))) \otimes \mathcal O_Y(\beta H). 
\end{align*}
Therefore, we have a homomorphism 
$$
\bigoplus_{\mathrm{finite}}(\mathcal O_Y(\beta H)\otimes g_*\mathcal O_{Y'})\to 
\widehat S^{2\alpha\beta}(f_*\mathcal O_X(k(K_{X/Y}+\Delta)))\otimes 
\mathcal O_Y(2\beta H)
$$
which is surjective over a nonempty Zariski open set. 
Note that $g_*\mathcal O_{Y'}\otimes \mathcal O_Y(\beta H)$ 
is generated by global sections. 
Therefore, $$\widehat S^{2\alpha\beta}(f_*\mathcal O_X(k(K_{X/Y}+\Delta)))\otimes 
\mathcal O_Y(2\beta H)$$ is generated by global sections over 
a nonempty Zariski open set. 
\end{step}
This means that $f_*\mathcal O_X(k(K_{X/Y}+\Delta))$ is weakly positive. 
\end{proof}

In the above proof of Theorem \ref{f-thm1.1}, we have already used the following 
lemma. It is a variant of Kawamata's cover (see \cite[Theorem 17]{kawamata1}). 
The description of Kawamata's covering trick in \cite[3.19.~Lemma]{ev} 
is very useful for our purpose. 
See also 
\cite[5.3.~Kawamata's covering]{ak} 
and \cite[Lemma 2.5]{viehweg-book}. 

\begin{lem}\label{f-lem8.3}
Let $f:X\to Y$ be a surjective morphism from a compact K\"ahler manifold 
$X$ to a smooth projective variety $Y$ and let $H$ be 
a Cartier divisor on $Y$. Let $d$ be an arbitrary positive integer. 
Then we can take a finite flat morphism $g:Y'\to Y$ from a smooth 
projective variety $Y'$ and a Cartier divisor $H'$ on $Y'$ such that 
$g^*H\sim dH'$ and that 
$X'=X\times _Y Y'$ is a compact K\"ahler manifold with 
$\omega_{X'/Y'}=\tau^*\omega_{X/Y}$, where $\tau:X'\to X$. 
Let $S$ be a simple normal crossing divisor on $X$. 
Then we can choose $g:Y'\to Y$ such that 
$\tau^*S$ is a simple normal crossing divisor on $X'$. 

Furthermore, let $D$ be a Cartier divisor on $X$. 
We put $f':X'\to Y'$. 
Then there is a natural isomorphism 
$$
f'_*\mathcal O_{X'}(nK_{X'/Y'}+\tau^*D)\simeq  g^*f_*\mathcal O_X(nK_{X/Y}+D)
$$ for every integer $n$. 
\end{lem}
\begin{proof}
We take general very ample Cartier divisors $D_1$ and $D_2$ with the following 
properties. 
\begin{itemize}
\item[(i)] $H\sim D_1-D_2$, 
\item[(ii)] $D_1$, $D_2$, $f^*D_1$, and $f^*D_2$ are smooth, 
\item[(iii)] $D_1$ and $D_2$ have no common components, and 
\item[(iv)] $\Supp (D_1+D_2)$ and $\Supp (f^*D_1+f^*D_2)$ 
are simple normal crossing divisors. 
\end{itemize}
We take a finite flat cover due to Kawamata 
with respect to $Y$ and $D_1+D_2$ (see \cite[Theorem 17]{kawamata1}), 
we obtain $g:Y'\to Y$ and $H'$ such that $g^*H\sim dH'$. 
By the construction of the above Kawamata cover $g:Y'\to Y$, 
we may assume that the ramification locus $\Sigma$ of $g$ in $Y$ is 
a general simple normal crossing divisor. This means that 
$f^*P$ is a smooth divisor for any irreducible component 
$P$ of 
$\Sigma$ and that $f^*\Sigma$ is a simple normal 
crossing divisor on $X$. We may further assume that 
$f^*P\not\subset S$ for any irreducible component $P$ of 
$\Sigma$ and that $f^*\Sigma+S$ is a simple normal crossing divisor 
on $X$ since $D_1$ and $D_2$ are general. 
In this situation, we can check that 
$X'=X\times _YY'$ is a compact K\"ahler manifold (see Remark \ref{f-rem2.12}) 
and that $\tau^*S$ is a simple 
normal crossing divisor on $X'$. 
$$
\xymatrix{
 X' \ar[r]^{\tau} \ar[d]_{f'} & X\ar[d]^{f} \\
  Y' \ar[r]_{g} & Y
} 
$$
By construction, we can also easily check that 
$\omega_{X'/Y'}=\tau^*\omega_{X/Y}$ by the Hurwitz formula. 
Therefore, we have 
$$
\mathcal O_{X'}(nK_{X'/Y'}+\tau^*D)\simeq \tau^*\mathcal O_X(nK_{X/Y}+D)
$$ 
for every integer $n$. 
Thus, we obtain 
\begin{align*}
f'_*\mathcal O_{X'}(nK_{X'/Y'}+\tau^*D) &\simeq f'_*\tau^*\mathcal O_X(nK_{X/Y}+D)
\\ &\simeq g^*f_*\mathcal O_X(nK_{X/Y}+D)
\end{align*} 
for every integer $n$. 
\end{proof}

As a special case of Theorem \ref{f-thm1.1}, we have: 

\begin{cor}[{cf.~\cite[Theorem III]{viehweg1}}]\label{f-cor8.4} 
Let $f:X\to Y$ be a surjective morphism such that 
$X$ is a complex manifold in Fujiki's class $\mathcal C$ and 
$Y$ is a smooth projective variety. 
Then $f_*\mathcal O_X(kK_{X/Y})$ is weakly positive for every $k>0$. 
\end{cor}

As is well-known, 
Corollary \ref{f-cor8.4} is a very famous fundamental result by Viehweg 
when $X$ is projective. 

\begin{rem} In Corollary \ref{f-cor8.4}, we did not check the weak positivity 
of $f_*\mathcal O_X(kK_{X/Y})$ in the sense of Definition \ref{f-def7.3}. 
Corollary \ref{f-cor8.4} says that $f_*\mathcal O_X(kK_{X/Y})$ is 
weakly positive in the sense of Definition \ref{f-def7.2}. 
\end{rem}

By Theorem \ref{f-thm1.1}, we can recover Campana's twisted 
weak positivity (see \cite[Theorem 4.13]{campana}). 

\begin{cor}
Let $f:X\to Y$ be a surjective morphism from a complex manifold 
$X$ in Fujiki's class $\mathcal C$ to a smooth projective 
variety $Y$. 
Let $D$ be a divisor on $X$. 
We put $D=D^h+D^v$ where $D^h$ $($resp.~$D^v$$)$ 
is the horizontal $($resp.~vertical$)$ part of $D$ with respect to $f:X\to Y$. 
Assume that $\Supp D^h$ is a simple normal crossing divisor 
and the coefficients of $D^h$ is less than or equal to $m$, 
where $m$ is a positive integer. 
Then $f_*\mathcal O_X(mK_{X/Y}+D)$ is weakly positive. 
\end{cor}

\begin{proof}
We put $\Delta=\frac{1}{m}D$. 
Then $(X, \Delta)$ is log canonical over the generic point of $Y$. 
We take a resolution $g:X'\to X$. 
Then 
$$
K_{X'}+\Delta'=g^*(K_X+\Delta)+E
$$
where $\Delta'$ and $E$ are effective and have no common components such that 
$\Supp (\Delta'+E)$ is a simple normal crossing divisor on $X'$. 
Let $\widetilde \Delta$ be the 
horizontal part of $\Delta'$. 
Then $(X', \widetilde \Delta)$ is log canonical. 
By Lemma \ref{f-lem7.7}, we can replace $(X, \Delta)$ with 
$(X', \widetilde \Delta)$. 
By Theorem \ref{f-thm1.1}, we obtain that 
$f_*\mathcal O_X(m(K_{X/Y}+\Delta))=f_*\mathcal O_X(mK_{X/Y}+D)$ 
is weakly positive. 
\end{proof}

\section{Addition formula}\label{f-sec9}

Theorem \ref{f-thm1.7} is an easy application of Theorem \ref{f-thm1.1}. 
It is contained in \cite{campana}. See also \cite{lu} and \cite{nakayama}. 

Proposition \ref{f-prop9.1} is a slight reformulation of 
\cite[Corollary 7.1]{viehweg1}. 

\begin{prop}[{cf.~\cite[Corollary 7.1]{viehweg1}}]\label{f-prop9.1}
Let $f:V\to W$ be an equidimensional surjective morphism 
from a normal projective variety $V$ to a smooth projective 
variety $W$ with connected fibers. 
Let $(V, \Delta)$ be a log canonical pair. 
Let $H$ be an ample Cartier divisor on $W$. 
Then there are some positive integers $a$ and $l$ such that 
$a(K_V+\Delta)$ is Cartier and 
the linear system $\Lambda$ associated 
to $$
H^0(V, \mathcal O_V(al(K_{V/W}+\Delta))\otimes f^*\mathcal O_W(lH))
$$ 
defines a rational map $\Phi:V\dashrightarrow X$ with 
$$\dim X=\kappa (V_w, K_{V_w}+\Delta|_{V_w})+\dim W,$$ where 
$V_w$ is a sufficiently general fiber of $f$. 
Moreover, there is a rational map $\pi:X\dashrightarrow W$ such that 
$f=\pi\circ \Phi$. 
$$
\xymatrix{
V\ar@{-->}[r]^{\Phi} \ar[d]_{f} & X\ar@{-->}[ld]^{\pi} \\
   W & 
} 
$$
\end{prop}

\begin{rem}\label{f-rem9.2}
When $\kappa (V_w, K_{V_w}+\Delta|_{V_w})=-\infty$, 
we claim nothing in Proposition \ref{f-prop9.1}. 
\end{rem}

\begin{proof}
We take a positive integer $a$ such that 
$a(K_V+\Delta)$ is Cartier and $f_*\mathcal O_V(a(K_{V/W}+\Delta))$ is nontrivial. 
By the twisted weak positivity theorem:~Theorem \ref{f-thm1.1}, there is some 
$b>0$ such 
that 
$$
\widehat{S}^{2b}(f_*\mathcal O_V(a(K_{V/W}+\Delta)))\otimes\mathcal O_W(b H)
$$ 
is generated by global sections on some nonempty Zariski open set. 
Since the 
natural map 
\begin{align*}
\widehat{S}^{2b}(f_*\mathcal O_V(a(K_{V/W}+\Delta)))\otimes \mathcal O_W(b H)
\\ \to f_*\mathcal O_V(2b a(K_{V/W}+\Delta))\otimes \mathcal O_W(b H)
\end{align*}
is nontrivial, we obtain an inclusion from 
$f^*\mathcal O_W(b H)$ into 
$\mathcal O_V(2b a(K_{V/W}+\Delta))\otimes f^*\mathcal O_W(2b H)$. 
Note that $f_*\mathcal O_V(2b a (K_{V/W}+\Delta))$ is reflexive since 
$f$ is equidimensional. 
Without loss of generality, we may assume that $b H$ is very ample by 
replacing $b$ with $b'b$ for some $b'\gg 0$. We put $l=2b$. 
If $b$ is sufficiently large, then the sections 
of $\mathcal O_V(2b a (K_{V/W}+\Delta))\otimes f^*\mathcal O_W(2b H)$ define 
a rational map $\Phi:V\dashrightarrow X$ such that 
$\mathbb C(X)$ is algebraically closed in $\mathbb C(V)$. 
Since $b H$ is very ample and there is 
an inclusion from  
$f^*\mathcal O_W(b H)$ into 
$\mathcal O_V(2b a(K_{V/W}+\Delta))\otimes f^*\mathcal O_W(2b H)$, 
we obtain a rational map $\pi:X\dashrightarrow W$ such that 
$f=\pi\circ \Phi$. 
The easy addition formula gives 
\begin{align*}
&\kappa (V_w, K_{V_w}+\Delta|_{V_w})\\&\leq \dim \Phi (V_w)+
\kappa(F, (\mathcal O_V(2b a(K_{V/W}+\Delta))\otimes f^*\mathcal O_W(2b 
H))|_F)\\&=\dim \Phi(V_w)
\end{align*} 
where $F$ is a sufficiently general fiber of $\Phi:V\dashrightarrow X$ 
(if necessary, we take an elimination of points of indeterminacy of $\Phi$). 
On the other hand, the restriction of the linear system $\Lambda$ to $V_w$ is a 
subsystem of $H^0(V_w, \mathcal O_{V_w}(2b a (K_{V_w}
+\Delta|_{V_w}))$. 
Therefore, $\dim \Phi (V_w) \leq \kappa (V_w, K_{V_w}+\Delta|_{V_w})$. 
Hence, 
we obtain $$\kappa (V_w, K_{V_w}+\Delta|_{V_w})=\dim \Phi(V_w)=
\dim X-\dim W. $$ 
\end{proof}

\begin{rem}\label{f-rem9.3}
In Proposition \ref{f-prop9.1}, it is sufficient to assume that 
$a$ is a positive integer such that $a(K_V+\Delta)$ is Cartier 
and that $f_*\mathcal O_V(a(K_{V/W}+\Delta))$ is nontrivial. 
\end{rem}

Let us start the proof of Theorem \ref{f-thm1.7}. 

\begin{proof}[Proof of Theorem \ref{f-thm1.7}] 
By \cite{ak}, we can construct a commutative diagram: 
$$
\xymatrix{
X' \ar[r]^{h} \ar[d]_{f'} & X\ar[d]^{f} \\
   Y' \ar[r]_{g} & Y
} 
$$
with the following properties: 
\begin{itemize}
\item[(i)] $f':X'\to Y'$ is an equidimensional surjective morphism 
from a normal projective variety $X'$ to a smooth projective 
variety $Y'$, 
\item[(ii)] $h$ and $g$ are birational, and 
\item[(iii)] $X'$ has only quotient singularities and 
$(U_{X'}\subset X')$ is toroidal for some nonempty Zariski open 
set $U_{X'}$. 
\end{itemize} 
By \cite{ak}, we may assume that 
$\Exc (h)\cup \Supp h^{-1}_*\Delta$ is contained in 
$X'\setminus U_{X'}$. 
We put 
$$
K_{X'}+\Delta'=h^*(K_X+\Delta)+E
$$ 
such that $(X', \Delta')$ is log canonical and 
$E$ is effective and $h$-exceptional. 
Let $H$ be an ample Cartier divisor on $Y'$. 
Since $K_{Y'}$ is a big divisor, 
by Kodaira's lemma, $aK_{Y'}\sim H+F$ for some effective divisor $F$ 
on $Y'$ and  a sufficiently divisible positive integer $a$. 
By Proposition \ref{f-prop9.1} (see also Remark \ref{f-rem9.3}), 
we have 
\begin{align*}
\kappa (X, K_X+\Delta)&=\kappa (X', K_{X'}+\Delta')\\
&\geq \kappa (X', a(K_{X'}+\Delta')-af'^*K_{Y'}+f'^*H)\\
&\geq \kappa (X'_y, K_{X'_y}+\Delta'|_{X'_y})+\dim Y'\\
&=\kappa (X_y, K_{X_y}+\Delta|_{X_y})+\dim Y. 
\end{align*}
Note that 
$$
\kappa (X, K_X+\Delta)\leq \kappa (X_y, K_{X_y}+\Delta|_{X_y})+\dim Y
$$ 
always holds by the easy addition formula. 
Therefore, we obtain 
$$
\kappa (X, K_X+\Delta)= \kappa (X_y, K_{X_y}+\Delta|_{X_y})+\dim Y. 
$$
This is the desired equality. 
\end{proof}

\section{Addition for logarithmic Kodaira dimensions}\label{f-sec10} 

We prove Theorem \ref{f-thm1.9}, which is due to Maehara (see \cite{maehara}), 
as an application of Theorem \ref{f-thm1.1}:~Twisted weak positivity. 
Before we start the proof of Theorem \ref{f-thm1.9}, 
we give some comments on Maehara's works (see \cite{maehara}). 

\begin{rem}\label{f-rem10.1}
In Theorem \ref{f-thm1.9}, 
$$f_*\mathcal O_X(k(K_X+D_X))\otimes \mathcal O_Y(-k(K_Y+D_Y))\otimes 
\mathcal O_Y(D_Y)$$ is weakly $1$-positive in the sense of Maehara for 
any $k>0$. 
It is the main theorem of \cite{maehara}. 
Maehara obtained Theorem \ref{f-thm1.9} as a corollary of the above 
weak $1$-positivity (see \cite[Corollary 2]{maehara}). 
In this section, we do not use Maehara's results and prove 
Theorem \ref{f-thm1.9} as an application of Theorem \ref{f-thm1.1} by 
using the weak semistable reduction theorem (see \cite{ak}).  
\end{rem}

Let us start the proof of Theorem \ref{f-thm1.9}. 
\begin{proof}[Proof of Theorem \ref{f-thm1.9}] 
By \cite{ak}, we may assume that 
\begin{itemize}
\item[(i)] $f: (U_X\subset X)\to (U_Y\subset Y)$ is toroidal 
and is equidimensional, 
\item[(ii)] $D_Y$ is contained in $Y\setminus U_Y$ 
and $D_X$ is contained in $X\setminus U_X$, 
\item[(iii)] $X$ has only quotient singularities, $Y$ is smooth, and 
\item[(iv)] $f$ is smooth over $U_Y$. 
\end{itemize}
Moreover, there is a Kawamata cover 
$\tau:Y'\to Y$ such that the normalization $X'$ of $\widetilde X=X\times_YY'$ 
is a weak semistable reduction over $Y'$.  Note that $f$ is flat. 
We put $\Delta_Y=Y\setminus U_Y$. 

\begin{lem}[{cf.~\cite[Main Theorem]{maehara}}]\label{f-lem10.2}
Under the above assumptions, 
$$
f_*\mathcal O_X(k(K_X+D_X))\otimes \mathcal O_Y(-k(K_Y+D_Y))\otimes \mathcal O_Y
(\Delta_Y)
$$
is weakly positive for 
every divisible positive integer $k$. 
\end{lem}
Once we establish Lemma \ref{f-lem10.2}, 
we can check: 
 
\begin{lem}\label{f-lem10.3} 
Let $H$ be an ample Cartier divisor on $Y$. 
Then there are some positive integers $a$ and $l$ such that $a(K_X+D_X)$ is 
Cartier and the linear system $\Lambda$ associated to 
$$
H^0(X, \mathcal O_X(al(K_{X/Y}+D_X-f^*D_Y))\otimes 
f^*\mathcal O_Y(lH+l\Delta_Y))
$$ 
defines a rational map $\Phi:X\dashrightarrow Z$ with 
$$
\dim Z=\kappa (F, K_F+D_X|_F)+\dim Y,
$$ 
where $F$ is a sufficiently general fiber of $f$. Moreover, 
there is a rational map $\pi:Z\dashrightarrow Y$ such that 
$f=\pi\circ \Phi$. 
\end{lem}
$$
\xymatrix{
X\ar@{-->}[r]^{\Phi} \ar[d]_{f} & Z\ar@{-->}[ld]^{\pi} \\
   Y & 
} 
$$
\begin{proof}
See the proof of Proposition \ref{f-prop9.1} (see also Remark \ref{f-rem9.3}). 
\end{proof}

If $a$ is a sufficiently large and divisible positive integer, then $a(K_X+D_X)$ is Cartier, 
$$f_*\mathcal O_X(a(K_X+D_X))\otimes \mathcal O_Y(-a(K_Y+D_Y))\otimes \mathcal O_Y(\Delta_Y)
\ne 0$$ is 
weakly positive by Lemma \ref{f-lem10.2}, and $$a(K_Y+D_Y)-\Delta_Y
\sim H+G$$ for an ample Cartier divisor $H$ 
and some effective divisor $G$ on $Y$ by Kodaira's lemma. 
Therefore, by Lemma \ref{f-lem10.3}, 
we obtain 
\begin{align*}
&\kappa (X, K_X+D_X)\\&\geq \kappa (X, a(K_X+D_X)-af^*(K_Y+D_Y)+f^*\Delta_Y+f^*H)
\\&\geq \kappa (F, K_F+D_X|_F)+\dim Y. 
\end{align*}
On the other hand, 
we always have 
$$
\kappa (X, K_X+D_X)\leq \kappa (F, K_F+D_X|_F)+\dim Y 
$$ 
by the easy addition formula. 
Thus, we obtain 
$$
\kappa (X, K_X+D_X)=\kappa (F, K_F+D_X|_F)+\dim Y. 
$$ 
This is the desired equality. 

Anyway, we see that it is sufficient to prove Lemma \ref{f-lem10.2}. 
For the proof of Lemma \ref{f-lem10.2}, 
we can replace $D_Y$ with $\Delta_Y$ 
and $D_X$ with $\Supp (D_X+f^*D_Y)$. 
Moreover, by adding some divisors to $D_Y$, we may further 
assume that $D_Y$ is a simple normal crossing 
divisor, 
$\Delta_Y\leq D_Y$, 
and $\tau$ is \'etale over $Y\setminus D_Y$. 
We put $f':X'\to Y'$, $p:X'\to \widetilde X$, $q:\widetilde X\to X$, and $K_{X'}+D_{X'}
=\lambda^*(K_X+D_X)$ where $\lambda=q\circ p: X'\to X$. 
$$
\xymatrix{
X'\ar[r]^{p}\ar[dr]_{f'}&\widetilde X \ar[r]^{q} \ar[d]^{\widetilde f} & X\ar[d]^{f} \\
  & Y' \ar[r]_{\tau} & Y
} 
$$
We also put $K_{Y'}+D_{Y'}=\tau^*(K_Y+D_Y)$. Note that 
$D_{X'}$ and $D_{Y'}$ are reduced. 
Let $\mathrm {Sing}(X)$ be the singular locus of $X$. 
We put $\Sigma =f^{-1}(f(\mathrm{Sing}(X)))$ and 
$\widetilde \Sigma=q^{-1}(\Sigma)$. 
Then $\codim _{X}\Sigma\geq 2$, 
$\codim _{\widetilde X}\widetilde{\Sigma}\geq 2$ and 
$\omega_{\widetilde X}|_{X^{\dag}}$ is locally free by the 
flat base change theorem (cf.~Theorem \ref{f-rem2.16}) 
with $X^{\dag}=\widetilde X
\setminus \widetilde{\Sigma}$. 
We put $\omega_{\widetilde X}^{[k]} =\iota_*((\omega_{\widetilde X}|_{X^\dag})^
{\otimes k})$ where $\iota: X^\dag\to \widetilde X$. 
We note that $\widetilde X$ is Cohen--Macaulay by the local description of 
Kawamata's cover. 
We also note that $\omega_{\widetilde X}^{[k]}$ is invertible if 
$\mathcal O_X(kK_X)$ is invertible. We can check: 

\begin{lem}[{cf.~\cite[Lemma C]{maehara}}]\label{f-lem10.4}
We have the following inclusion 
\begin{align*}\tag{$\spadesuit$}\label{f-shiki2}
&p_*\mathcal O_{X'}(k(K_{X'/Y'}+D_{X'}-f'^*D_{Y'}))
\\&\subset \omega_{\widetilde X/Y'}^{[k]}\otimes 
\mathcal O_{\widetilde X}(k(q^*(D_X-f^*D_{Y})))\otimes
\mathcal O_{\widetilde X}(q^*f^*\Delta_Y)
\end{align*} 
for every divisible positive integer $k$. 
\end{lem}
\begin{proof}[Proof of Lemma \ref{f-lem10.4}] 
Since $k$ is divisible, 
we have that $\omega^{[k]}_{\widetilde X/ Y'}$ is invertible and 
that $k(K_{X'/Y'}+D_{X'}-f'^*D_{Y'})$ and 
$k(q^*(D_X-f^*D_Y))$ are Cartier. 
Therefore, the both hand sides satisfy Serre's $S_2$ condition. 
We also note that 
$X\setminus \Sigma$ is smooth and that 
$X^\dag$ is Gorenstein. 
By replacing $X$ and $\widetilde X$ 
with $X\setminus \Sigma$ and $X^\dag$ respectively, 
we may assume that $X$ is smooth and that $\widetilde X$ is 
Gorenstein. 
We have 
$$
p_*\mathcal O_{X'}(K_{X'})\subset \omega_{\widetilde X}
$$ 
because $p$ is the normalization. 
Since $D_{X'}-\lambda^*D_X\leq 0$, we obtain 
$$
p_*\mathcal O_{X'}(K_{X'/Y'}+D_{X'}-\lambda^*D_X)\subset \omega_{\widetilde X/Y'}. 
$$ 
This is equivalent to 
\begin{align*}
&p_*\mathcal O_{X'}(K_{X'/Y'}+D_{X'}-f'^*D_{Y'})\otimes \mathcal O_{\widetilde X}(\widetilde f^*D_{Y'})\\ 
&\subset \omega_{\widetilde X/Y'}\otimes \mathcal O_{\widetilde X}(q^*(D_X-f^*D_Y))\otimes \mathcal O_{\widetilde X}
(q^*f^*D_Y). 
\end{align*} 
Since $\widetilde f^*D_{Y'}$ is effective, 
we have 
\begin{align*}
&p_*\mathcal O_{X'}(K_{X'/Y'}+D_{X'}-f'^*D_{Y'})\\ 
&\subset \omega_{\widetilde X/Y'}\otimes \mathcal O_{\widetilde X}(q^*(D_X-f^*D_Y))\otimes \mathcal O_{\widetilde X}
(q^*f^*D_Y). 
\end{align*} 
Over the generic point of every irreducible component 
of $D_Y-\Delta_Y$, $f$ is smooth, 
$\widetilde f$ is smooth, 
$p$ is an isomorphism, 
and $D_X-f^*D_Y=D_{X'}-f'^*D_{Y'}=0$. 
Therefore, we obtain 
\begin{align*}\tag{$\heartsuit$}\label{f-shiki3}
&p_*\mathcal O_{X'}(K_{X'/Y'}+D_{X'}-f'^*D_{Y'})\\ 
&\subset \omega_{\widetilde X/Y'}\otimes \mathcal O_{\widetilde X}(q^*(D_X-f^*D_Y))\otimes \mathcal O_{\widetilde X}
(q^*f^*\Delta_Y). 
\end{align*}
We note that 
\begin{align*}
q^*\mathcal O_X(K_X+D_X-f^*(K_Y+D_Y))
\simeq \omega_{\widetilde X/Y'}\otimes \mathcal O_{\widetilde X}(q^*(D_X-f^*D_Y))
\end{align*} 
by the flat base change theorem (cf.~Theorem \ref{f-rem2.16}) and 
\begin{align*}
&p^*q^*\mathcal O_X(K_X+D_X-f^*(K_Y+D_Y))
\\&\simeq \mathcal O_{X'}(K_{X'}+D_{X'}-f'^*(K_{Y'}+D_{Y'})). 
\end{align*}
By taking 
$$\otimes \omega^{\otimes k-1}_{\widetilde X/Y'}\otimes \mathcal O_{\widetilde X}
((k-1)q^*(D_X-f^*D_Y))$$ with (\ref{f-shiki3}), we obtain the desired inclusion by the projection formula. 
\end{proof}

Let us go back to the proof of Lemma \ref{f-lem10.2}. 
Note that 
\begin{align*}
&\omega_{\widetilde X/Y'}^{[k]}\otimes 
\mathcal O_{\widetilde X}(k(q^*(D_X-f^*D_{Y})))\\ 
&\simeq
q^*(\mathcal O_X(k(K_X+D_X))\otimes f^*\mathcal O_Y(-k(K_Y+D_Y))) 
\end{align*} 
when $k$ is a divisible positive integer. 
By applying $\widetilde f_*$, 
(\ref{f-shiki2}) implies: 
\begin{lem}\label{f-lem10.5}
There exists a generically isomorphic inclusion 
\begin{align*}
&f'_*\mathcal O_{X'}(k(K_{X'}+D_{X'}))\otimes \mathcal O_{Y'}(-k(K_{Y'}+D_{Y'}))
\\ &\subset 
\tau^*(f_*\mathcal O_X(k(K_X+D_X))
\otimes \mathcal O_Y(-k(K_Y+D_Y))\otimes \mathcal O_Y(\Delta_Y))
\end{align*}
for every divisible positive integer $k$. 
\end{lem}This is because $g$ is flat. 
Since 
\begin{align*}
&f'_*\mathcal O_{X'}(k(K_{X'}+D_{X'}))\otimes \mathcal O_{Y'}(-k(K_{Y'}+D_{Y'}))
\\ 
&=f'_*\mathcal O_{X'}(k(K_{X'/Y'}+D_{X'}-f'^*D_{Y'}))
\end{align*}
is weakly positive by Theorem \ref{f-thm1.1}, 
we can easily check that 
$$f_*\mathcal O_X(k(K_X+D_X))\otimes \mathcal O_Y(-k(K_Y+D_Y))\otimes \mathcal 
O_Y(\Delta_Y)
$$ is also weakly positive (see, for example, \cite[Lemma 1.4.~5)]{viehweg1} and 
\cite[Lemma 3.6]{fujino-revisited}). 
Note that 
$D_{X'}-f'^*D_{Y'}$ is effective 
since $f'$ is weakly semistable. 
Thus we obtain Lemma \ref{f-lem10.2}. 
This implies that the equality in Theorem \ref{f-thm1.9} holds. 
\end{proof}

Note that Theorem \ref{f-thm1.9} contains a generalization of 
\cite[Theorem 30]{kawamata1}, which plays an important role 
for Kawamata's theorem on the {\em{quasi-Albanese maps}} for varieties of the logarithmic 
Kodaira dimension zero 
(see \cite[Corollary 29]{kawamata1}). For the details, 
see \cite{fujino-quasi-alb}. 


\end{document}